\documentclass[11pt]{amsart}

\usepackage[utf8]{inputenc}
\usepackage[margin=1in]{geometry}
\usepackage{amsfonts}
\usepackage{amsmath}
\usepackage{amssymb}
\usepackage{amsthm}
\usepackage[dvipsnames]{xcolor}
\usepackage{tikz}
\usepackage{multirow}
\usepackage{float}
\usepackage{mathtools}
\usepackage{parskip}
\usepackage{afterpage}
\usepackage{relsize} 
\usepackage{float}
\usepackage{graphicx} 
\usepackage{subcaption}
\usetikzlibrary{cd}
\usetikzlibrary{decorations.pathreplacing}
\usetikzlibrary{decorations.pathmorphing}
\usetikzlibrary{decorations.markings}

\theoremstyle{definition}
\newtheorem{defn}{Definition}[section]

\newtheorem{thm}[defn]{Theorem}
\newtheorem{prop}[defn]{Proposition}
\newtheorem{lem}[defn]{Lemma}

\newtheorem{ex}[defn]{Example}

\newtheorem{prob}{Problem}

\DeclareMathOperator{\ind}{\mathbf{ind}}

\DeclareMathOperator{\dis}{d_{\mathfrak{g}}}

\renewcommand{\a}[1]{\mathfrak{#1}} 

\newcommand{\cl}[1]{\mathrm{#1}} 
\newcommand{\CK}{\cl{K}} 

\newcommand{\defeq}{\stackrel{\mathclap{\textsf{\tiny{def}}}}{=}}

\newcommand\add{\stackrel{\mathclap{\mathsmaller{\mathsmaller{\langle\bullet\rangle}}}}{\longrightarrow}}
\newcommand\addn[1]{\stackrel{\mathclap{\mathsmaller{\mathsmaller{\langle\bullet\rangle}}}}{\longrightarrow}_{\times #1}}

\newcommand\dda{\stackrel{\mathclap{\mathsmaller{\mathsmaller{\langle\bullet\rangle}}}}{\longleftarrow}}

\newcommand{\ie}{i.e.\ }
\newcommand{\eg}{e.g., }
\newcommand{\cf}{cf.\ }

\newcommand{\Core}[1]{\mathfrak{C}_{#1}} 
\newcommand{\Mpl}[2]{\mathfrak{C}_{#2,#1}} 
\newcommand{\MuN}{\mathfrak{C}_{\omega}} 

\DeclareMathOperator{\MGen}{d} 

\newcommand{\bigslant}[2]{{\raisebox{.2em}{$#1$}\left/\raisebox{-.2em}{$#2$}\right.}}

\title{On the networks of large embeddings}
\thanks{This research is support by the Hungarian National Research, Development and Innovation Office -- NKFIH, grant no.\ FK 134732.}
\author{Tu\u{g}ba Aslan}
\address{Tu\u{g}ba Aslan,
Medipol Business School, Istanbul Medipol University, Istanbul, Turkey.}
\email{tugba.aslankhalifa@medipol.edu.tr}
\author{Mohamed Khaled}
\address{Mohamed Khaled, Faculty of Engineering and Natural Sciences, Bah\c{c}e\c{s}ehir University, Istabul, Turkey.}
\email{mohamed.khalifa@eng.bau.edu.tr}
\author{Gergely Sz\'ekely}
\address{Gergely Sz\'ekely, Alfr\'ed R\'enyi Institute of Mathematic, Budapest, Hungary \& University of Public Service, Budapest, Hungary.}
\email{szekely.gergely@renyi.hu}

\subjclass[2000]{Primary 08A60, 03C05; Secondary 08A05, 05C12}
\keywords{large embeddings; generator distance; monounary algebras}

\allowdisplaybreaks

\begin{document}

\begin{abstract}
We define a special network that exhibits the large embeddings in any class of similar algebras. With the aid of this network, we introduce a notion of distance that conceivably counts the minimum number of dissimilarities, in a sense, between two given algebras in the class in hand; with the possibility that this distance may take the value $\infty$. We display a number of inspirational examples from different areas of algebra, \eg group theory and monounary algebras, to show that this research direction can be quite remarkable.
\end{abstract}

\maketitle

\section{Introduction}

We assume familiarity with the most basic conceptions of universal algebra; the notions of algebra, subalgebra, homomorphism, isomorphism, embedding, direct product, etc. A modest perceptiveness to some special areas of algebra is also required, \eg group theory, fields, lattices, Boolean algebras, etc. This paper is written with the appetite of reducing the technical difficulties and the passion of creating a text accessible to mathematicians with general background. There is a considerable number of illustrative figures that can demonstrate some concepts and ideas of proofs.

Throughout, unless otherwise stated, we assume that $\CK$ is an arbitrary but fixed class of similar algebras. We define a special network that exhibits the large embeddings in the class $\CK$. An algebra $\a{A}$ is largely embeddable into another algebra $\a{B}$ in $\CK$ if{}f there is a member of $\CK$ that serves as an intermediate algebra in the following sense: It is isomorphic to $\a{A}$, and it is a large subalgebra of $\a{B}$. The word `large' here means that this intermediate algebra needs at most one extra element to generate the whole of $\a{B}$. Similar networks can be found in \cite{ksz} and \cite{kszlf}.

Investigating this network seems to be thought-provoking and promising to be fruitful. For instance, with the aid of this network, we introduce a notion of distance that conceivably counts the minimum number of dissimilarities, in a sense, between two given structures in $\mathrm{K}$; with the possibility that this distance may take the value $\infty$. Formally, this distance counts the minimum number of steps needed from the symmetric closure of the large embedding relation to reach an algebra from another algebra. This notion of distance provides a framework that can give a qualitative and quantitative analysis of the connections between the algebraic structures in the class $\CK$.

In this paper, we commence a research project that aims to investigate the network and the distance notions for different classes of algebras. We give several examples from different areas of algebra, \eg group theory, lattice theory, etc. As an example, we give a complete characterization of the network of all subgroups of the group $\bigslant{\mathbb{Q}}{\mathbb{Z}}$. The last section is completely devoted to the network of monounary algebras. For instance, we give an algorithmic method to determine the distance between any two finite monounary algebras. Similar investigations concerning Boolean algebras and other algebras of logic are burgeoning \cite{ksz}, \cite{LC21Mohamed} and \cite{LC21Gergely}. 

We foresee that this open-end research direction would appeal to mathematicians from different disciplines, and that many captivate results can be obtained. We are also discerned with the potential applications that this project may activate. Here is an example: If we suppose that $\CK$ is the class of all Lindenbaum--Tarski algebras of first-order theories, then the notion of distance we propose here brings a method to determine the concepts distinguishing two theories in hand. This idea could be used to conclude that only one concept distinguishes classical and relativistic kinematics, namely the existence of observers who are at absolute rest, see \cite{kszlf} and \cite{LSz18}. This is indeed an interesting result for logicians, philosophers and physicists alike.

{\textbf{On the notations}:} We usually use the German capital letters
$\a{A}, \a{B}, \a{C}, \ldots$ to denote algebras, and we use the
corresponding Latin capital letters $A,B,C,\ldots$ to denote their
underlying sets. However, we may violate this rule when we give
examples from areas of algebra where such convention is not used, \eg
group theory. We use the notation $\a{A}\subseteq\a{B}$ to convey that
$\a{A}$ is a subalgebra of $\a{B}$. For an algebra $\a{A}$ and a
subset $X\subseteq A$, we let $\langle X\rangle$ denote the subalgebra
of $\a{A}$ generated by $X$. We consider $0$ as a natural number and denote the set of natural numbers by
$\omega$.

\section{The networks of large embeddings}

We start with the central notion of the present paper, the notion of \emph{large subalgebras}.

\begin{defn}
Suppose that $\a{A}$ is a subalgebra of $\a{B}$. We say that
\emph{$\a{A}$ is a \textbf{large subalgebra} in $\a{B}$} if{}f
there is an element $b\in B$ such that $\langle A\cup\{b\}\rangle=\a{B}$.\footnote{This notion was inspected in some other resources in the literature, \cf \cite{vujosevic}.}
\end{defn}

For an example, Table~\ref{fig:symmetricgroup} on page
\pageref{fig:symmetricgroup} characterizes all the large subgroups of
the symmetric group $S_4$. This table also shows that a large proper
subalgebra is not necessarily maximal among proper subalgebras, \eg
$\{e,(12)\}$ is large but not maximal in $S_4$. However, the converse
is always true; every subalgebra which is maximal among proper
subalgebras is large. The table also shows an example of two
isomorphic subgroups, namely $\{e,(12)\}$ and $\{e,(12)(34)\}$, where
the former is large in $S_4$ and the latter is not large in $S_4$.

\begin{table}[!ht]
\renewcommand{\arraystretch}{1.4}
\begin{tabular}{||c|l|c|c||}
\hline
\hline 
$n$ & Subgroups of $S_4$ of order $n$ & $\cong$ type & Large in $S_4$?\\
\hline
\hline 
1 & $\{e\}$ & $Z_1$ & NO\\
\hline
\hline 
\multirow{9}{*}{2} & $\{e,(12)\}$ & \multirow{6}{*}{$Z_2$} & \multirow{6}{*}{YES}\\
&$\{e,(13)\}$ &&\\ 
& $\{e,(14)\}$ &&\\ 
&$\{e,(23)\}$ &&\\ 
&$\{e,(24)\}$ &&\\
&$\{e,(34)\}$ && \\
\cline{2-4}
& $\{e,(12)(34)\}$ & \multirow{3}{*}{$Z_2$} & \multirow{3}{*}{NO} \\ 
& $\{e,(13)(24)\}$ &&\\ 
& $\{e,(14)(23)\}$ && \\
\hline
\hline 
\multirow{4}{*}{3} & $\{e,(123),(132)\}$ & \multirow{4}{*}{$Z_3$} &\multirow{4}{*}{YES}\\ 
&$\{e,(124),(142)\}$ &&\\
&$\{e,(134),(143)\}$ &&\\
&$\{e,(234),(243)\}$ && \\
\hline
\hline 
\multirow{7}{*}{4} & $\{e,(12),(34),(12)(34)\}$ & \multirow{3}{*}{$Z_2\times Z_2$} & \multirow{3}{*}{YES}\\
&$\{e,(13),(24),(13)(24)\}$ &&\\
& $\{e,(14),(23),(14)(23)\}$ && \\
\cline{2-4}
&  $\{e,(1324),(12)(34),(1423)\}$ & \multirow{3}{*}{$Z_4$} & \multirow{3}{*}{YES} \\
& $\{e,(1234),(13)(24),(1432)\}$ && \\
&$\{e,(1243),(14)(23),(1342)\}$ && \\
\cline{2-4}
& $\{e,(12)(34),(13)(24),(14)(23)\}$  & $Z_2\times Z_2$ & NO \\
\hline
\hline
\multirow{4}{*}{6} &$\{e,(123),(132),(12),(13),(23)\}$ & \multirow{4}{*}{$S_3$} & \multirow{4}{*}{YES}\\ & $\{e,(124),(142),(12),(14),(24)\}$ &&\\
&$\{e,(134),(143),(13),(14),(34)\}$ &&\\ &$\{e,(234),(243),(23),(24),(34)\}$ &&\\
\hline
\hline 
\multirow{3}{*}{8}&$\{e,(12),(34),(12)(34),(13)(24),(14)(23),(1324),(1423)\}$& \multirow{3}{*}{$D_4$} & \multirow{3}{*}{YES}\\
&$\{e,(13),(24),(13)(24),(12)(34),(14)(23),(1234),(1432)\}$ &&\\
&$\{e,(14),(23),(14)(23),(12)(34),(13)(24),(1243),(1342)\}$ &&\\
\hline
\hline
\multirow{2}{*}{12} & $\big\{e,(12)(34),(13)(24),(14)(23), (123),(132),(124), (142),(134),$ & \multirow{2}{*}{$A_4$} & \multirow{2}{*}{YES} \\ & \hspace{8cm} $(143),
(234),(243)\big\}$   && \\
\hline
\hline 
24 & $S_4$ & $S_4$ & YES\\
\hline
\hline 
\end{tabular}
\caption{Large subgroups of the symmetric group $S_4$\label{fig:symmetricgroup}}
\end{table}

Now, we construct a special network which is convened from the large inclusions in $\CK$. This network is a special case of the cluster networks defined in \cite[section 3]{kszlf}. The \emph{\textbf{network $\mathcal{N}(\CK)$ of large  embeddings} in $\CK$} is defined to be the graph whose vertices are the members of $\CK$, and which has two types of edges: {\color{red}{red dashed}} edges connecting the isomorphic members of $\CK$ and {\color{blue}{blue}} edges connecting two algebras of $\CK$ if one of them is a large subalgebra of the other one.\footnote{We note that the definition of the network $\mathcal{N}(\CK)$ is not formalizable in Zermelo–Fraenkel set theory if $\CK$ is a proper class. However, all our results can be formulated within von Neumann–Bernays–G\"odel set theory (NBG). One cannot define ordered pairs of proper classes even in NBG. But this does not bother us because we do not really need ordered pairs here. All our definitions can be understood as follows: “for classes $x$, $y$, etc., having certain properties there are classes $z$, etc., such that...”. We use the terminology of graphs only for the sake of simplicity.}

Every vertex in the network $\mathcal{N}(\CK)$ has two loops; one is {\color{red}{red dashed}}, and the other is {\color{blue}{blue}}. When we illustrate networks of large embeddings, we will omit these loops for the lucidity of the illustrations. 

\begin{ex}
Suppose that $\CK$ is the class of all subalgebras of the $8$-element Boolean algebra $\a{A}$ having atoms $a$, $b$ and $c=\lnot(a\lor b)$. Then the network of $\CK$ is illustrated in Figure~\ref{fig:BA}.

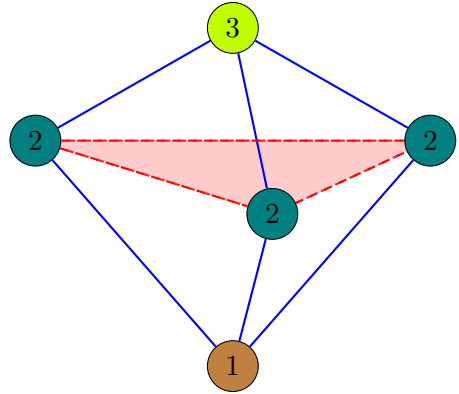
\begin{figure}[!ht]
\begin{minipage}{0.38\textwidth}
 \tikz\node[circle,draw,fill=brown,scale=0.8]{$1$}; The $2$-element subalgebra $\{0,1\}$ \\
        \tikz\node[circle,draw,fill=teal,scale=0.8]{$2$}; Subalgebras having $4$-elements:
        
        \begin{itemize}
            \item generated by $a$: $\{0,a,\lnot a,1\}$
            \item generated by $b$: $\{0,b,\lnot b,1\}$
            \item generated by $c$: $\{0,c,\lnot c,1\}$
        \end{itemize}
        
    \tikz\node[circle,draw,fill=lime,scale=0.8]{$3$}; The whole algebra $\a{A}$ \\

	\tikz\path[red,thick,dashed] [-] (0,0) edge (1,0); Red dashed edges\\
	\tikz\path[blue,thick] [-] (0,0) edge (1,0); Blue edges\\

\end{minipage}
\begin{minipage}{0.6\textwidth}
\centering
\begin{tikzpicture}[scale=0.75]
\begin{scope}
\filldraw[thick,red!90,dashed,fill=red!20] (-1.5,0) -- (5.5,0) -- (2.7,-1.3) -- cycle;
    \node[circle,draw,fill=teal] (D) at (-1.5,0) {$2$};
    \node[circle,draw,fill=teal] (A) at (5.5,0) {$2$};
    \node[circle,draw,fill=teal] (B) at (2.7,-1.3) {$2$};
    \node[circle,draw,fill=brown] (C) at (2,-4) {$1$};
    \node[circle,draw,fill=lime] (V) at (2,2) {$3$};
    \path[blue,thick] [-] (A) edge (C);
    \path[blue,thick] [-] (B) edge (C);
    \path[blue,thick] [-] (D) edge (C);
   \path[blue,thick] [-] (A) edge (V);
    \path[blue,thick] [-] (B) edge (V);
    \path[blue,thick] [-] (D) edge (V);
 \path[red,thick,dashed] [-] (D) edge (A);
 \path[red,thick,dashed] [-] (D) edge (B);
 \path[red,thick,dashed] [-] (B) edge (A);
\end{scope}
\end{tikzpicture}
\end{minipage}
\caption{A small network of Boolean algebras}
\label{fig:BA}
\end{figure}
\end{ex}

\begin{ex}
Suppose that $\CK$ is the class of all subgroups of the alternating group $A_4$. Then the network
$\mathcal{N}(\CK)$ is illustrated in Figure~\ref{fig:s4}.

\begin{figure}[!ht]
\begin{minipage}{0.38\textwidth}
 \tikz\node[circle,draw,fill=brown,scale=0.8]{$1$}; Trivial subgroup \\
        \tikz\node[circle,draw,fill=violet,scale=0.8]{$2$}; Subgroups isomorphic to $Z_2$ \\
        \tikz\node[circle,draw,fill=teal,scale=0.8]{$3$}; Subgroups isomorphic to $Z_3$ \\
        \tikz\node[circle,draw,fill=cyan,scale=0.8]{$4$}; Subgroup isomorphic to $Z_2\times Z_2$ \\
        \tikz\node[circle,draw,fill=lime,scale=0.8]{$5$}; The alternating group $A_4$ \\

	\tikz\path[red,thick,dashed] [-] (0,0) edge (1,0); Red dashed edges\\
	\tikz\path[blue,thick] [-] (0,0) edge (1,0); Blue edges\\

\end{minipage}
\begin{minipage}{0.6\textwidth}
\centering
\begin{tikzpicture}[scale=0.6]
\filldraw[thick,red!90,dashed,fill=red!20] (0,-0.6) -- (-2,0) -- (1,0) -- cycle;
\fill[red!90,fill=red!20] (-5,-4) -- (6,-4) -- (4,-6) -- (-7,-6) -- cycle;
     \node[circle,draw,fill=violet,scale=0.8] (D) at (-2,0) {$2$};
     \node[circle,draw,fill=violet,scale=0.8] (A) at (1,0) {$2$};
     \node[circle,draw,fill=violet,scale=0.8] (B) at (-0.2,-0.6) {$2$};
     \node[circle,draw,fill=brown,scale=0.8] (C) at (-0.5,-2) {$1$};
     \node[circle,draw,fill=cyan,scale=0.8] (V) at (-0.6,1.5) {$4$};
    \path[blue,thick] [-] (A) edge (C);
    \path[blue,thick] [-] (B) edge (C);
    \path[blue,thick] [-] (D) edge (C);
   \path[blue,thick] [-] (A) edge (V);
    \path[blue,thick] [-] (B) edge (V);
    \path[blue,thick] [-] (D) edge (V);
    
\node[circle,draw,fill=teal,scale=0.8] (S1) at (-5,-4) {$3$};
\node[circle,draw,fill=teal,scale=0.8] (S2) at (6,-4) {$3$};
\node[circle,draw,fill=teal,scale=0.8] (S3) at (4,-6) {$3$};
\node[circle,draw,fill=teal,scale=0.8] (S4) at (-7,-6) {$3$};
\path[red,thick,dashed] [-] (S1) edge (S2);
\path[red,thick,dashed] [-] (S1) edge (S3);
\path[red,thick,dashed] [-] (S1) edge (S4);
\path[red,thick,dashed] [-] (S3) edge (S2);
\path[red,thick,dashed] [-] (S4) edge (S2);
\path[red,thick,dashed] [-] (S3) edge (S4);

\path[blue,thick] [-] (C) edge (S1);
\path[blue,thick] [-] (C) edge (S2);
\path[blue,thick] [-] (C) edge (S3);
\path[blue,thick] [-] (C) edge (S4);

\node[circle,draw,fill=lime,scale=0.8] (T) at (-0.5,5) {$5$};
\path[blue,thick] [-] (T) edge (S1);
\path[blue,thick] [-] (T) edge (S2);
\path[blue,thick] [-] (T) edge (S3);
\path[blue,thick] [-] (T) edge (S4);
\path[blue,thick] [-] (T) edge (A);
\path[blue,thick] [-] (T) edge (B);
\path[blue,thick] [-] (T) edge (D);
\path[blue,thick] [-] (T) edge (V);
\end{tikzpicture}
\end{minipage}
\caption{A small network of groups}\label{fig:s4}
\end{figure}
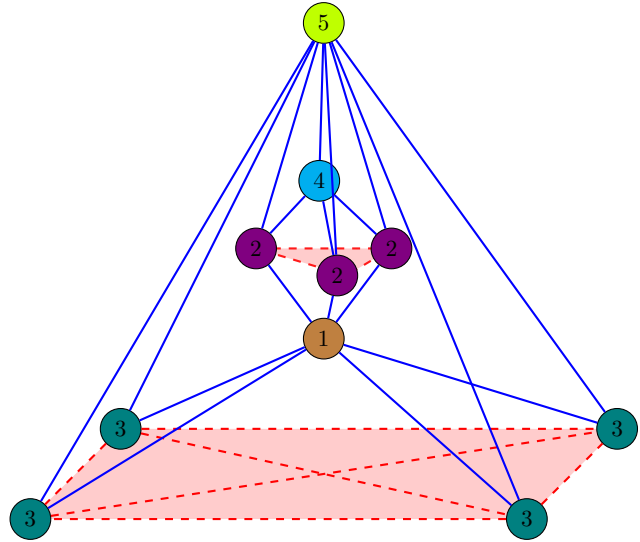
\end{ex}

A \emph{\textbf{path}} in the network $\mathcal{N}(\CK)$ is a finite sequence of edges such that any two consecutive edges share a common vertex. The \emph{\textbf{{\color{blue}{blue}}-length} of a path} is defined to be the number of its blue edges. Two algebras $\a{A}$ and $\a{B}$ are said to be \emph{\textbf{connected} in $\mathcal{N}(\CK)$} if there is a path connecting them. 

\begin{ex}If $\CK$ is the class of all finite fields, then two finite fields are connected in $\mathcal{N}(\CK)$ if{}f they have the same characteristic. 
\end{ex}

\begin{defn}
The \emph{\textbf{generator distance}} $\dis:\CK\times\CK\rightarrow \omega\cup\{\infty\}$ is defined on the class $\CK$ as follows: Let $\a{A}$ and $\a{B}$ be two algebras in $\CK$. 
\begin{itemize}
\item If $\a{A}$ and
$\a{B}$ are not connected in $\mathcal{N}(\CK)$, then
$\dis(\a{A},\a{B})\defeq \infty$. \item Otherwise, if $\a{A}$ and
$\a{B}$ are connected in $\mathcal{N}(\CK)$, then $$\dis(\a{A},\a{B})\defeq\min\left\{ \ l:\text{ there is a path of {\color{blue}{blue}}-length $l$ connecting $\a{A}$ and $\a{B}$ in the network $\mathcal{N}(\CK)$} \ \right\}.$$
\end{itemize}
\end{defn}

Another way to understand the above definition: Each {\color{blue}{blue}} edge represents an actual step, a step of adding or deleting a generator. The {\color{red}{red dashed}} edges represent null steps, steps that are not in our consideration; moving from an algebra to an isomorphic copy represents no real `algebraic' change. The generator distance then counts the minimum number of actual steps needed to reach an algebra starting from another one. 

\begin{ex}\label{ex-vectorspaces}
Suppose that $\CK$ is the class of all vector spaces over a given
field, then the generator distance between two vector spaces $V$ and
$W$ in $\CK$ is the following:
\begin{equation*}
  \dis(V, W)=\begin{cases}
    0 & \text{ if } V\cong W, \\
    |dim(V)-dim(W)| &  \text{ if $V$ and $W$ are both finite dimensional, and}\\
    \infty & \text{ otherwise}.
  \end{cases}
\end{equation*}
\end{ex}

\begin{ex}\label{ex:v4q8}
Suppose that $\CK$ is the class of all groups, and let $V_4$ and $Q_8$ denote the Klein four group and the quaternion group, respectively. It is obvious that $\dis(V_4,Q_8)\geq 2$ because none of these groups is isomorphic to or embeddable into the other one. We claim that the equality holds. To see this, let $G$ be the central product of the dihedral group $D_8$ and the cyclic group $Z_4$. It is not hard to check  that each one of $V_4$ and $Q_8$ is isomorphic to a large subgroup of $G$.
\end{ex}


\begin{prop}
The following are true for every $\a{A},\a{B},\a{C}\in \CK$:
\begin{enumerate}
\item $\dis(\a{A},\a{B})\geq 0$, and $\dis(\a{A},\a{B})=0\iff \a{A}\cong\a{B}$.
\item $\dis(\a{A},\a{B})=\dis(\a{B},\a{A})$.
\item $\dis(\a{A},\a{B})\leq \dis(\a{A},\a{C}) + \dis(\a{C},\a{B})$, where addition with $\infty$ is defined in the natural way.
\end{enumerate}
\end{prop}
\begin{proof}
Straightforward.
\end{proof}

The following simple observation is worthy of note. Suppose that $\cl{L}$ is a subclass of $\CK$. The values of the generator distance in $\cl{L}$ must take bigger or equal values to the corresponding values of the generator distance in $\CK$. This is true because every path in $\mathcal{N}(\cl{L})$ is a path in $\mathcal{N}(\CK)$, but the later network may contain extra paths that may reduce the values of the generator distance.

\begin{defn}
Let $\a{A}$ and $\a{B}$ be two algebras in $\CK$. We say that \emph{$\a{A}$ is \textbf{largely embeddable} into $\a{B}$ \textbf{in $\CK$}}, and we write $\a{A}\add\a{B}$ or $\a{B}\dda\a{A}$, if{}f there is $\a{C}\in\CK$ such that $\a{A}\cong\a{C}$ and $\a{C}$ is a large subalgebra in $\a{B}$. 
\end{defn}

There are subgroups that are
not large in $S_4$, but every non-trivial subgroup of $S_4$ is largely
embeddable into $S_4$, see Table~\ref{fig:symmetricgroup}. In \cite[Thm.A]{IZ95}, it was show that, if $n\neq 4$ and $x\in S_n$ not the identity element, then there is $y\in S_n$ such that $x$
and $y$ generates the whole $S_n$. Consequently, every non-trivial
subgroup of $S_n$ is large if $n\neq 4$.  Now, suppose that $\CK$ is the class of all groups, and let $G$ and $H$ be two finite groups.
\begin{itemize}
    \item The generator distance between $G$ and $H$ in $\CK$ is at most $2$. This is true because there is always a (large enough) natural number $n$ such that both $G$ and $H$ are embeddable into the symmetric group $S_n$. Whence, by \cite[Thm.A]{IZ95} and the triangle inequality, it follows that
    \begin{equation*}\label{eq:gptriangle}
  \dis(G,H)\leq \dis(G,S_n)+\dis(S_n,H)\leq 1+1=2.
\end{equation*}
This bound is sharp by Example~\ref{ex:v4q8}.
\item For the same reasons, one can see that the distance between $G$ and the trivial group is at most $3$. This bound is also sharp and that can be confirmed, for example, by checking that the quaternion group is of distance $3$ from
the trivial group.
\end{itemize}

In the illustrative diagrams of this section, a black edge between two algebras means that one of these algebras is largely embeddable into the other one. When we are interested in the direction of the large embedding (which one of the algebras is largely embeddable into the other one) we will use a black arrow instead of the black edge. Again, we discard all the loops.

Assume that the class $\CK$ is closed under the formation of subalgebras. Thus, the generator distance between two algebras is equal to the minimum length of all paths of black edges that connect these algebras. To find a path of black edges of length $k$, for some $k\in\omega$, we have $2^k$-many possible options for the arrows (regardless of the different options for the vertices); which makes the problem of determining the generator distance quite difficult. The following properties may reduce this difficulty substantially. Indeed, with these properties, one needs to consider only $k+1$-many options for the directions of the arrows to find a path of length $k$.

\begin{figure}[!ht]
\begin{minipage}{0.45\textwidth}
\centering
\begin{tikzpicture}[scale=0.7, >=latex]
            \centering
                \draw[->,thick] (1.5,-1) -- (0,1) ;
                \draw[->,thick] (1.5,-1) -- (3,1);

                \draw[->,thick,dashed] (0,1) -- (1.5,3);
                \draw[->,thick,dashed] (3,1) -- (1.5,3) node[black]{$\bullet$} node[black,above] {$\mathfrak{D}$};
                    \node[black] at (1.5,-1) {$\bullet$};
                    \node[black,below] at (1.5,-1) {$\mathfrak{C}$};
                    \node[black] at (0,1) {$\bullet$};
                    \node[black,left] at (0,1) {$\mathfrak{A}$};
                    \node[black] at (3,1) {$\bullet$};
                    \node[black,right] at (3,1) {$\mathfrak{B}$};
\end{tikzpicture}
\subcaption{Push-up property}\label{fig:pushup}
\end{minipage}
\begin{minipage}{0.45\textwidth}
\centering
\begin{tikzpicture}[scale=0.7, >=latex]
            \centering
                \draw[->,thick,dashed] (1.5,-1) -- (0,1) ;
                \draw[->,thick,dashed] (1.5,-1) -- (3,1);

                \draw[->,thick] (0,1) -- (1.5,3);
                \draw[->,thick] (3,1) -- (1.5,3) node[black]{$\bullet$} node[black,above] {$\mathfrak{D}$};
                    \node[black] at (1.5,-1) {$\bullet$};
                    \node[black,below] at (1.5,-1) {$\mathfrak{C}$};
                    \node[black] at (0,1) {$\bullet$};
                    \node[black,left] at (0,1) {$\mathfrak{A}$};
                    \node[black] at (3,1) {$\bullet$};
                    \node[black,right] at (3,1) {$\mathfrak{B}$};
\end{tikzpicture}
\subcaption{Push-down property}\label{fig:pushdown}
\end{minipage}
\caption{Pushing properties}\label{fig:pushing}
\end{figure}
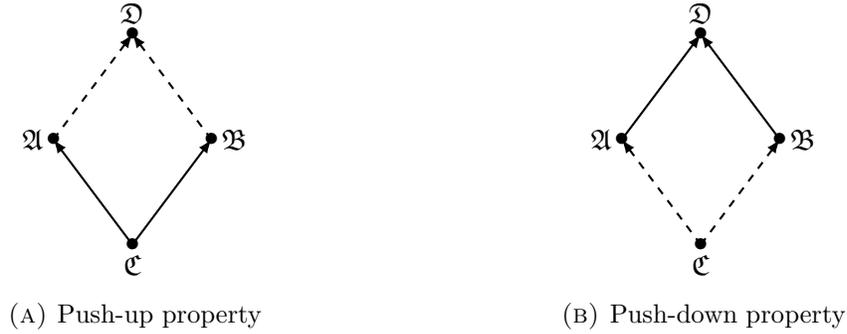

\begin{defn}
We say that the class $\CK$ has the
\emph{\textbf{push-up property}} if{}f for all
$\a{A},\a{B}\in\CK$, we have
  \begin{equation}\label{eq-amalg1}
    (\exists \a{C}\in\CK)\quad \a{A}\dda \a{C}\add
    \a{B}\ \implies\ (\exists \a{D}\in\CK)\quad \a{A}\add\a{D}\dda
    \a{B}.
  \end{equation}
We say that the class $\CK$ has the
\emph{\textbf{push-down property}} if{}f for all
$\a{A},\a{B}\in\CK$, we have
  \begin{equation}\label{eq-2amalg1}
    (\exists \a{D}\in\CK)\quad \a{A}\add\a{D}\dda
    \a{B}\ \implies\ (\exists \a{C}\in\CK)\quad \a{A}\dda\a{C}\add
    \a{B}.
  \end{equation}
\end{defn}

To state the next proposition, we will use the following
abbreviation: For any algebras $\a{A}$ and $\a{B}$ in the class $\CK$, and
$m\in\omega$, we write $\a{A}\addn{m}\a{B}$ if{}f there are
$\a{A}_0,\ldots,\a{A}_{m}$ such that
\begin{equation*}
\a{A}\cong\underbrace{\a{A}_0\add\a{A}_1\add\ldots
   \add \a{A}_{m}}_{m\text{-many arrows}}\cong\a{B}.
\end{equation*}

\begin{prop}\label{prop:pushpath}
Suppose that $\CK$ is closed under formation of subalgebras. Suppose
that $\a{A}$ and $\a{B}$ are two algebras connected in
$\mathcal{N}(\CK)$. The following are true:
\begin{enumerate}
\item If $\CK$ has the push-up property, then
  \begin{equation*}
 (\exists \a{D}\in\CK)
    (\exists n,m\in\omega) \ \ \ \a{A}\addn{m}\a{D},\ \a{B}\addn{n}\a{D} \text{ and } \dis(\a{A},\a{B})=n+m.
  \end{equation*}
\item If $\CK$ has the push-down property, then
   \begin{equation*}
 (\exists \a{C}\in\CK)
    (\exists n,m\in\omega) \ \ \ \a{C}\addn{m}\a{A},\ \a{C}\addn{n}\a{B} \text{ and } \dis(\a{A},\a{B})=n+m.
  \end{equation*}
\end{enumerate}
\end{prop}
\begin{proof}
The statements follow from the definitions by straightforward induction on $\dis(\a{A},\a{B})$ which is finite because $\a{A}$ and $\a{B}$ are connected in $\mathcal{N}(\CK)$. See an illustrative proof in Figure~\ref{fig:pushpath}.\qedhere
\end{proof}

\begin{figure}[H]
    \centering
    \begin{tikzpicture}[scale=0.75, >=latex]
      \draw[blue!50, line width=3] (0,0) to node[black,below=-1,sloped]{$1$}
      (1,1) to node[black,above=-1,sloped]{$2$} (2,0) to
      node[black,above=-1,sloped]{$\dots$} (3,-1) to (4,0) to (5,-1) to (6,0)
      to (7,-1) to node[black,above=-1,sloped]{$\dots$} (8,0) to
      node[black,above=-1,sloped]{\tiny $\dis(\a{A},\a{B})$} (9,1);

      \foreach \x in {0,1,2,3,4}{
        \foreach \y in {0,1,2,3,4,5}{
          \ifnum \y=5 \else \draw[->]
          (\x+\y,\x-\y) -- (\x+\y+1,\x-\y+1);
          \fi
          \ifnum \x=4 \else
          \draw[->] (6+\x-\y,4-\x-\y)-- (5+\x-\y,5-\x-\y); \fi
        }
      }

  \node[below]  at (4,-4) {$\a{C}$};
      \node[above]  at (5,5) {$\a{D}$};
      \node[left] at (0,0){$\a{A}$};
      \node[right] at (9,1){$\a{B}$};

      \draw [decorate,decoration={brace,amplitude=10pt,raise=2}]
      (0,0) -- (5,5) node [black,midway,above left=7] 
            {$m$};

            \draw [decorate,decoration={brace,amplitude=10pt,raise=2}]
            (5,5) -- (9,1) node [black,midway,above right=7] 
                  {$n$};

        \draw [decorate,decoration={brace,mirror,amplitude=10pt,raise=2}] (0,0) -- (4,-4) node [black,midway,below left=8]       {$n$};

                    \draw [decorate,decoration={brace,amplitude=10pt,raise=2}]
     
                        (9,1) -- (4,-4) node [black,midway,below right=7] 
                              {$m$};
    \end{tikzpicture}
    \caption{Pushing a path upwards and/or downwards \label{fig:pushpath}}
  \end{figure}
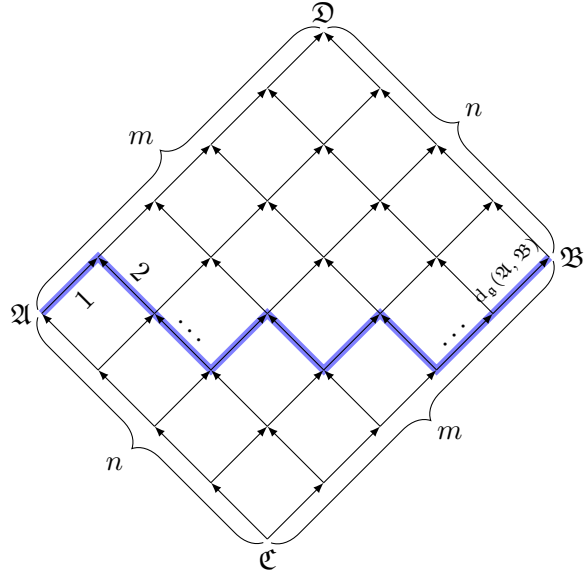

We recall an important property from the literature. The class $\CK$
is said to have \emph{the \textbf{amalgamation property}} if{}f for
every $\a{A},\a{B},\a{C}\in\CK$ and every embeddings
$f:\a{C}\rightarrow\a{A}$ and $g:\a{C}\rightarrow\a{B}$ there is
$\a{D}\in\CK$ and embeddings $f':\a{A}\rightarrow\a{D}$ and
$g':\a{B}\rightarrow\a{D}$ such that $f'\circ f=g'\circ g$. For more
on different versions of the amalgamation property and their
applications, we refer the reader to \cite{KMPT83}.

One should not confuse the push-up property with the amalgamation property. The arrows in Figure~\ref{fig:pushup} represent large embeddings, not arbitrary embeddings, and the commutativity of the diagram in Figure~\ref{fig:pushup} is not required. These are the exact differences between the push-up property and the amalgamation property.

\begin{prop}\label{prop:ap->pup}
Suppose that $\CK$ is closed under the formation of subalgebras, then
$$\CK\text{ has the amalgamation property }\implies \CK \text{ has the push-up property}.$$
\end{prop}

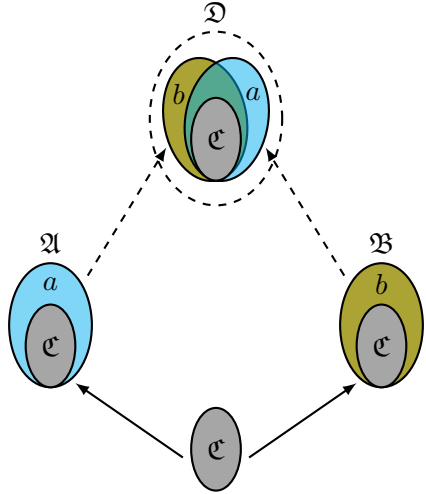
\begin{figure}[!ht]
        \begin{tikzpicture}[scale=1.1, >=latex]
            \centering
            \node[above] at (0,1.05) {$\mathfrak{D}$};
            \draw[thick,dashed] (0,0) ellipse (0.8 cm and 1.05cm);
            \draw[thick,rotate around={10:(0,-0.75)},fill=olive,fill opacity=0.8] (0,0) ellipse (0.5cm and 0.75cm);
            \draw[thick,rotate around={350:(0,-0.75)},fill=cyan,fill opacity=0.5] (0,0) ellipse (0.5cm and 0.75cm);
            \draw[thick,fill=gray!70] (0,-0.25) ellipse (0.3cm and 0.5cm);
            \node at (0,-0.25) {$\mathfrak{C}$};
            \node at (0.45,0.3) {$a$};
            \node at (-0.45,0.3) {$b$};
            \node[above] at (-2,-1.75) {$\mathfrak{A}$};
            \draw[thick,draw,fill=cyan,fill opacity=0.5] (-2,-2.5) ellipse (0.5cm and 0.75cm);
            \draw[thick,draw,fill=gray!70] (-2,-2.75) ellipse (0.3cm and 0.5cm);
            \node at (-2,-2) {$a$};
            \node at (-2,-2.75) {$\mathfrak{C}$};
            \node[above] at (2,-1.75) {$\mathfrak{B}$};
            \draw[thick,draw,fill=olive,fill opacity=0.8] (2,-2.5) ellipse (0.5cm and 0.75cm);
            \draw[thick,draw,fill=gray!70] (2,-2.75) ellipse (0.3cm and 0.5cm);
            \node at (2,-2) {$b$};
            \node at (2,-2.75) {$\mathfrak{C}$};
\draw[thick,fill=gray!70] (0,-4) ellipse (0.3cm and 0.5cm);
            \node at (0,-4) {$\mathfrak{C}$};
                \draw[thick,->] (0.4,-4.1) -- (1.7,-3.2);
                \draw[thick,->] (-0.4,-4.1) -- (-1.7,-3.2);
                 \draw[thick,->,dashed] (-1.55,-1.9) -- (-0.6,-0.35);
                \draw[thick,->,dashed] (1.55,-1.9) -- (0.6,-0.35);
        \end{tikzpicture}
\caption{Amalgamation property $\implies$ push-up property}\label{apimpliespup}
\end{figure}

\begin{proof}
Assume that $\CK$ has the amalgamation property. We need to show that $\CK$ has the push-up property. To this end, let $\a{A},\a{B},\a{C}\in\CK$ and suppose that $\a{C}$ is largely embeddable into both $\a{A}$ and $\a{B}$. Thus, there are two embeddings $f:\a{C}\rightarrow\a{A}$ and $g:\a{C}\rightarrow\a{B}$, $a\in A$ and $b\in B$ such that 
\begin{equation*}
\a{A}=\langle f[C]\cup\{a\}\rangle \ \ \ \text{ and } \ \ \ \a{B}=\langle g[C]\cup\{b\}\rangle.\footnote{Here, $f[X]$ denotes the $f$-image of set $X$, \ie $f[X]=\{f(x): x\in X\}$.}
\end{equation*}
Now, by the amalgamation property, there is an algebra $\a{D}\in\CK$ and embeddings $f':\a{A}\rightarrow\a{D}$ and $g':\a{B}\rightarrow\a{D}$ such that $f'\left[f[C]\right]=g'\left[g[C]\right]\defeq X$. Consider the following subalgebras of $\a{D}$:
\begin{equation*}
\a{D}_1=\langle X \cup\{f'(a)\}\rangle, \ \ \a{D}_2=\langle X \cup\{g'(b)\}\rangle \ \ \text{ and } \ \ \a{D}'=\langle X \cup\{f'(a),g'(b)\}\rangle.
\end{equation*}
Since $\CK$ is closed under formation of subalgebras, we guarantee that $\a{D}_1$, $\a{D}_2$ and $\a{D}'$ are all members of $\CK$. It is not hard to see that 
$\a{A}\cong\a{D}_1$ and $\a{B}\cong\a{D}_2$, and both $\a{D}_1$ and $\a{D}_2$ are large subalgebras in $\a{D}'$. Therefore, we conclude that the class $\CK$ has the push-up property. See Figure~\ref{apimpliespup}.
\end{proof}

There are several classes of algebras that have the amalgamation
property, and hence the push-up property, such as groups
\cite{S27,N48}, (distributive) lattices \cite{J57,GJL73,P68}, Boolean
algebras \cite{DY63}, some classes of Boolean algebras with operators
\cite{MSA07,N85,P71}. Thus, the difficulty of the problem of
determining the values of the generator distance in these classes can
be reduced. In what follows, we give an example of how
Proposition~\ref{prop:pushpath} and Proposition~\ref{prop:ap->pup} can
be effectively used in this direction. We also note that some other
classes do not have the amalgamation property, such as semigroups
\cite[section 9.4]{Clifford-Preston67} and modular lattices
\cite{jonsson71,GJL73}, and so Proposition~\ref{prop:ap->pup} cannot
be useful in these cases.

Consider the group $G\defeq\bigslant{\mathbb{Q}}{\mathbb{Z}}$ and suppose that $\CK$ is the class of all subgroups of this group. It is known that $G$ is isomorphic to the torsion subgroup of the unit circle. In other words,  $G$ is isomorphic to the group that consists of all $p^n$-th complex roots of the unit, for each prime number $p$ and each $n\in\omega$, and whose operation is the multiplication of complex numbers. See Figure~\ref{fig:unit circle}. The group $G$ has many interesting properties, we list some of them below.  

 \begin{figure}[!ht]
            \begin{tikzpicture}[scale=0.8]
                \pgfmathsetmacro{\r}{2} 
                \pgfmathsetmacro{\l}{2.4} 

                \draw[gray]  (-\l,0) to (\l,0) (0,-\l) to (0,\l);

                \draw[magenta, ultra thick] (0,0) circle [radius=\r];

                \foreach \x in {0,1,...,7}{
                    \node[circle, fill,inner sep=2] (\x) at (45*\x:\r) {};
                }

                \node[below right] at (0) {$z^0=1$};
                \node[above right] at (1) {$z^1$};
                \node[above left] at (2) {$z^2$};
                \node[above left] at (3) {$z^3$};
                \node[below left] at (4) {$z^4$};
                \node[below left] at (5) {$z^5$};
\node[below right] at (6) {$z^6$};
\node[below right] at (7) {$z^7$};
                \draw[blue,ultra thick] (0) to (1) to (2) to (3) to (4) to (5) to (6) to (7) to (0) ;
            \end{tikzpicture}
\caption{The $2^3$-th complex roots of the unit}\label{fig:unit circle}
           \end{figure}
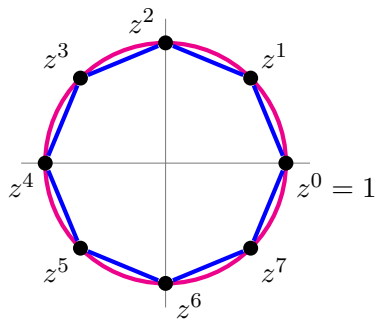

\begin{enumerate}
\item[(A)] $G$ is Abelian and torsion. The later means that each element of $G$ has a finite order.
\item[(B)] Every finitely generated subgroup of $G$ is cyclic.
\item[(C)] Any two isomorphic subgroups of $G$ must be equal. Thus, $\CK$ has the amalgamation property, since the group $G$ itself can be always chosen as the amalgamating algebra.
\item[(D)] Let $p$ be a prime number and let $n\in\omega$. The group $G$ contains a cyclic group of order $p^n$, denoted by $Z(p^n)$. These cyclic subgroups form a chain whose limit gives another important subgroup of $G$. This subgroup is called the Pr\"ufer $p$-group and is denoted by $Z(p^{\infty})$.
$$Z(p^0)\subseteq Z(p^1)\subseteq Z(p^2)\subseteq Z(p^3)\subseteq\cdots \hspace{1cm} \subseteq Z(p^{\infty}).$$ 
\end{enumerate}

Let $p_0,p_1,\ldots,p_j,\ldots$ be an enumeration of the primes in their natural order. Consider an infinite sequence of the form $\bar{k}=(k_0,k_1,\ldots,k_j,\ldots)$ where each $k_i$ is either a non-negative integer or the symbol $\infty$. We call such a sequence a \emph{\textbf{choice sequence}}. Let 
$$\langle \bar{k}\rangle=\left\{g_1g_2\cdots g_m\in G \ \bigg\vert \ 0\not=m\in\omega\text{ and }\{g_1,g_2,\ldots,g_m\}\subseteq\bigcup_{i=0}^{\infty} Z(p_i^{k_i})\right\}.$$
Obviously, $\langle\bar{k}\rangle$ is a subgroup of $G$. Moreover, it is not hard to see that the following is true, \cf \cite{BEAUMONT-ZUCKERMAN51}. 
\begin{enumerate}
\item[(E)] Every subgroup of $G$ is equal to $\langle\bar{k}\rangle$ for some choice sequence $\bar{k}=(k_0,k_1,\ldots,k_j,\ldots)$. 
\end{enumerate}

Let $\bar{k}$ and $\bar{k}'$ be two choice sequences. We need the following notations.
\begin{itemize}
\item We write $\bar{k}\equiv\bar{k}'$ if{}f the following two conditions are satisfied: 
\begin{enumerate}
\item For each $ i\in\omega$, $k_i=\infty\iff k'_i=\infty$, and
\item the set $\{i\in\omega\ \vert \ k_i\not=k'_i\}$ is finite. 
\end{enumerate}
\item We write $\bar{k}\preceq\bar{k}'$ if{}f (1) $\bar{k}\equiv\bar{k}'$ and (2) $k_i\leq k'_i$, for each $i\in\omega$.
\end{itemize}

Now, we are ready to give a complete characterization of the network $\mathcal{N}(\CK)$ of large embeddings and the generator distance in the class $\CK$ of all subgroups of $G=\bigslant{\mathbb{Q}}{\mathbb{Z}}$.

\begin{lem}\label{lem:q/z}
Let $\bar{k}$ and $\bar{k}'$ be two choice sequences.
\begin{enumerate}
\item If $\bar{k}=\bar{k}'$, then $\dis\left(\langle\bar{k}\rangle,\langle\bar{k}'\rangle\right)=0$.
\item If $\bar{k}\preceq\bar{k}'$ and $\bar{k}\not=\bar{k}'$, then $\dis\left(\langle\bar{k}\rangle,\langle\bar{k}'\rangle\right)=1$.
\item If $\bar{k}\equiv\bar{k}'$, $\bar{k}\npreceq\bar{k}'$ and $\bar{k}'\npreceq\bar{k}$, then $\dis\left(\langle\bar{k}\rangle,\langle\bar{k}'\rangle\right)=2$.
\item If $\bar{k}\not\equiv\bar{k}'$, then $\dis\left(\langle\bar{k}\rangle,\langle\bar{k}'\rangle\right)=\infty$.
\end{enumerate}
\end{lem}
\begin{proof}
Let $\bar{k}$ and $\bar{k}'$ be two choice sequences.
\begin{enumerate}
    \item If $\bar{k}=\bar{k}'$, then $\langle\bar{k}\rangle=\langle\bar{k}'\rangle$ and consequently $\dis\left(\langle\bar{k}\rangle,\langle\bar{k}'\rangle\right)=0$.
    \item Suppose that $\bar{k}\preceq\bar{k}'$ and $\bar{k}\not=\bar{k}'$. Then, $\langle\bar{k}\rangle\subseteq \langle\bar{k}'\rangle$, indeed, for each $i\in\omega$, $Z(p^{k_i})\subseteq Z(p^{k'_i})$. Remember, the set $\Delta\defeq \{i\in\omega\ \vert \ k_i\not=k'_i\}$ is finite. For each $j\in \Delta$, let $a_j$ be a generator of the group $Z(p^{k'_j})$. Let
    $$a\defeq \prod_{j\in \Delta} a_j.$$
    Now, it is not hard to see that adding $a$ to $\langle\bar{k}\rangle$ would generate the bigger group $\langle\bar{k}'\rangle$, which means that $\langle\bar{k}\rangle$ is a large subgroup in $\langle\bar{k}'\rangle$ and the desired follows.
    \item Suppose that $\bar{k}\equiv\bar{k}'$, $\bar{k}\npreceq\bar{k}'$ and $\bar{k}'\npreceq\bar{k}$. That means there are $i,j\in\omega$ such that $k_i< k'_i$ and $k_j> k'_j$. Thus, one can find an element of order $p^{k_j}$ in $\langle \bar{k}\rangle$ but not in $\langle \bar{k}'\rangle$. Similarly, one can find an element of order $p^{k'_i}$ in $\langle \bar{k}'\rangle$ but not in $\langle \bar{k}\rangle$. Hence, none of the groups $\langle \bar{k}\rangle$ and $\langle \bar{k}'\rangle$ cannot be embeddable into the other one. That implies $\dis\left(\langle\bar{k}\rangle,\langle\bar{k}'\rangle\right)\geq 2$. On the other hand, the set $\{i\in\omega\ \vert \ k_i\not=k'_i\}$ is finite. Define a choice sequence $\bar{l}=(l_0,l_1,\ldots)$ as follows: for each $i\in\omega$, let $l_i=\max\{k_i,k'_i\}$. Therefore, by the triangle inequality and item (2) above,
$$\dis\left(\langle\bar{k}\rangle,\langle\bar{k}'\rangle\right)\leq \dis\left(\langle\bar{k}\rangle,\langle\bar{l}\rangle\right)+\dis\left(\langle\bar{l}\rangle,\langle\bar{k}'\rangle\right)=1+1=2.$$
\item Suppose that $\bar{k}\not\equiv\bar{k}'$, then we have two cases.
\begin{enumerate}
    \item[Case 1:] Suppose that there is an $i\in\omega$ such that $k_i=\infty$ but $k'_i\not=\infty$. Assume towards a contradiction that $\dis\left(\langle\bar{k}\rangle,\langle\bar{k}'\rangle\right)<\infty$. By property (C) above, we know that the class $\CK$ has the amalgamation property, and whence, by Proposition~\ref{prop:ap->pup}, $\CK$ has the push-up property too. So, by Proposition~\ref{prop:pushpath}, there are $n,m\in\omega$ and a group $H\subseteq G$ such that $\langle\bar{k}\rangle\addn{m}H$, $\langle\bar{k}'\rangle\addn{n}H$ and $\dis\left(\langle\bar{k}\rangle,\langle\bar{k}'\rangle\right)=n+m$. We also note that $n\not=0$ since the group $\langle\bar{k}\rangle$ cannot be embedded into the other group $\langle\bar{k}'\rangle$. Now we consider the embedding $\langle\bar{k}'\rangle\addn{n}H$. 
    
    Note that property (C) implies that the group $\langle\bar{k}'\rangle$ is actually a subgroup of $H$. Moreover, there are $x_1,\ldots,x_n\in G$ such that $\langle\bar{k}'\rangle$ together with $x_1,\ldots,x_n$ generates the whole $H$. We can assume that none of these elements is a member of $\langle\bar{k}'\rangle$; otherwise we would get $\dis\left(\langle\bar{k}\rangle,\langle\bar{k}'\rangle\right)<n+m$ which is a contradiction. By the commutativity of $G$, it follows that $H$ is the internal product of $\langle\bar{k}'\rangle$ and the subgroup $\langle x_1,\ldots,x_n\rangle$ generated by $x_1,\ldots,x_n$. Now, the property $(B)$ implies that $\langle x_1,\ldots,x_n\rangle$ is a cyclic group. Thus, there is an element $x\in G$ such that $H$ is the internal product of $\langle\bar{k}'\rangle$ and the cyclic group $\langle x\rangle$ generated by the element $x$. Let
    $$l=\max\left(\{k'_i\}\cup\{s:p_i^s\text{ divides the order of }x\}\right).$$
    Since all these groups are Abelian, the order of every element in $H$ is equal to the least common multiple of the orders of an element in $\langle\bar{k}'\rangle$ and an element in $\langle x\rangle$. Thus, $H$ cannot contain an element of order $p_i^{l+1}$. But $\langle\bar{k}\rangle$ contains an element of order $p_i^{l+1}$. This contradicts the embedding $\langle\bar{k}\rangle\addn{m}H$. Whence, the desired follows.
  \item[Case 2:] Suppose that $\{i\in\omega \ \vert
    \ k_i\not=k'_i\}$ is infinite. Without loosing generality, we can
    assume that $k_i>k_i'$ holds for infinitely many
    $i\in\omega$. By the same argument (by contradiction) as in
    Case 1, there has to be an element $x\in G$ such that
    $\langle\bar{k}\rangle\addn{m}H$ holds for the internal product
    $H$ of $\langle\bar{k}'\rangle$ and $\langle x\rangle$. From the
    infinitely many, there has to be one, say $j$, such that
    $k_j>k_j'$ and the order of $x$ is not divisible by
    $p_j$. Therefore, $H$ cannot contain an element of order
    $p_j^{k_j}$ contradicting that $\langle\bar{k}\rangle$ can be
    embedded to $H$. \qedhere
\end{enumerate}
\end{enumerate}
\end{proof}

We note that the relation $\equiv$ is an equivalence relation on the set of all choice sequences. Let $\bar{k}$ be a choice sequence and let 
$$\CK(\bar{k})=\left\{\langle \bar{k}'\rangle\subseteq G \ \big\vert \  \bar{k}'\equiv\bar{k}\right\}.$$
By Lemma~\ref{lem:q/z}, the values of the generator distance between the members of $\CK(\bar{k})$ and each other must be finite, while the generator distance between a subgroup in $\CK(\bar{k})$ and another subgroup outside $\CK(\bar{k})$ is $\infty$. That makes the subclass $\CK(\bar{k})$ a connected component of $\mathcal{N}(\CK)$. By \emph{a \textbf{connected component} of a network of large embeddings}, we mean a maximal subclass in which any two algebras are connected in the network. 

\begin{figure}[!ht]
\centering
\begin{tikzpicture}[line cap = round, line join = round,]
\draw (-1,-1) -- (5,-1) -- (5,2.5) -- (-1,2.5) -- (-1,-1);
\draw (0,0) node {$\bullet$};
\draw (1,0) node {$\bullet$};
\draw (0,0) -- (1,0);
\draw (0.5,0) ellipse (1 and 0.5);
\draw (0.5,1.5) ellipse (1 and 0.5);
\draw (3.5,1.5) ellipse (1 and 0.5);
\draw (3,1.3) node {$\bullet$};
\draw (3.5,0) ellipse (1 and 0.5);
\path[every node/.style={sloped,anchor=south,auto=false},dashed]
        (1,0) edge              node {\tiny{infinite}} (3,1.3) ;
 \path[every node/.style={sloped, anchor=south,auto=false}]
        (0,0) edge              node {\tiny{finite}} (1,0) ;
\end{tikzpicture}
\caption{Connected components of a network}
\end{figure}
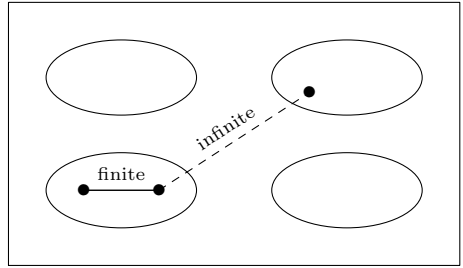

\emph{The \textbf{diameter} of a connected component} of a network of large embeddings is defined to be the smallest $n\in\omega\cup\{\infty\}$ for which $\dis(\a{A},\a{B})\leq n$ for all $\a{A}$ and $\a{B}$ in the component. So, if we assume again that $\CK$ is the class of all subgroups of $G=\bigslant{\mathbb{Q}}{\mathbb{Z}}$, then here are the diameters of the components of the network $\mathcal{N}(\CK)$: Let $\bar{k}=(k_0,k_1,\ldots,k_j,\ldots)$ be a choice sequence.
\begin{itemize}
\item If $\left\vert\left\{i\in\omega \ \big\vert \ k_i\not=\infty\right\}\right\vert=0$, then the diameter of $\CK(\bar{k})$ is $0$.
\item If $\left\vert\left\{i\in\omega \ \big\vert \ k_i\not=\infty\right\}\right\vert=1$, then the diameter of $\CK(\bar{k})$ is $1$.
\item If $\left\vert\left\{i\in\omega \ \big\vert \ k_i\not=\infty\right\}\right\vert>1$, then the diameter of $\CK(\bar{k})$ is $2$.
\end{itemize}
In other words, the network $\mathcal{N}(\CK)$ of all subgroups of $\bigslant{\mathbb{Q}}{\mathbb{Z}}$ is a union of infinitely many connected component, each of which is of diameter $0$, $1$ or $2$.

The implication in Proposition~\ref{prop:ap->pup} cannot be replaced by an equivalence. Here is a straightforward example. Suppose that $L=\{0,a,b,1\}$ is the $4$-element Boolean lattice and assume that $\CK$ is the class of all sublattices of $L$. Up to isomorphism, there are only $4$ sublattices of $L$, so it is not hard to see that $\CK$ in this case has the push-up property\footnote{It is customary in universal algebra to assume that the underlying set of an algebra is not empty. We follow this tradition here, and thus there is no empty lattice.}. We show that $\CK$ does not have the amalgamation property. We do that by contradiction. Assume that $\CK$ has the amalgamation property. Consider the sublattice $L'=\{0,1\}$ and the following two embeddings of $L'$ into $L$:

\begin{minipage}{0.45\textwidth}
\centering
$$
\begin{array}{ccc}
L' & \stackrel{f}{\longrightarrow} & L \\
0 & \mapsto & 0\\
1 & \mapsto & 1
\end{array}
$$
\end{minipage}
\begin{minipage}{0.45\textwidth}
\centering
$$
\begin{array}{ccc}
L' & \stackrel{g}{\longrightarrow} & L \\
0 & \mapsto & 0\\
1 & \mapsto & a
\end{array}
$$
\end{minipage}

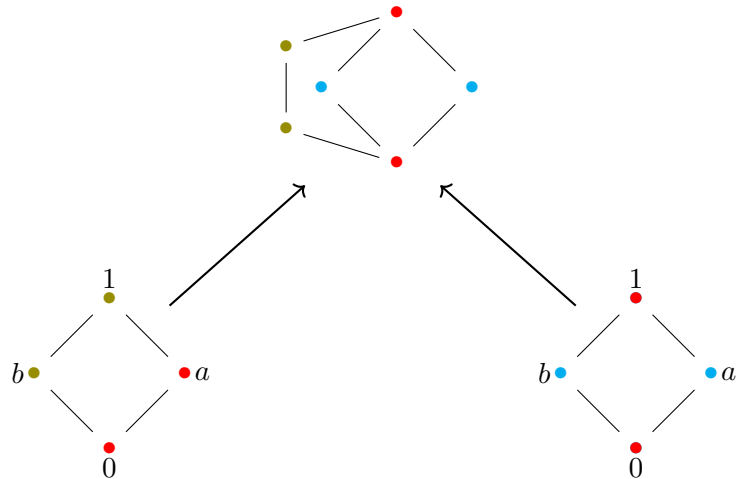
\begin{figure}[!ht]
\begin{tikzpicture}  
        \node[red] at (1,-1) (C) {$\bullet$};
        \node[red] at (0,-2) (A) {$\bullet$};       
        \node[olive] at (-1,-1) (B) {$\bullet$};
        \node[olive] at (0,0) (D) {$\bullet$};
        \draw (A) -- (B) node[left]{$b$};
        \draw (C) -- (A) node[below]{$0$};
        \draw (B) -- (D) node[above]{$1$};
        \draw (D) -- (C) node[right]{$a$};
        \node[cyan] at (7,-2) (AA) {$\bullet$};
        \node[cyan] at (7,0) (DD) {$\bullet$};
        \node[red] at (7,-2) (AA) {$\bullet$};
        \node[red] at (7,0) (DD) {$\bullet$}; 
        \node[cyan] at (6,-1) (BB) {$\bullet$};
        \node[cyan] at (8,-1) (CC) {$\bullet$};
        \draw (AA) -- (BB)  node[left]{$b$};
        \draw (CC) -- (AA)  node[below]{$0$};
        \draw (BB) -- (DD)  node[above]{$1$};
        \draw (DD) -- (CC)  node[right]{$a$};
    \node[red] at (3.82,1.8) (AAA) {$\bullet$};
        \node[cyan] at (2.82,2.8) (BBB) {$\bullet$};
        \node[cyan] at (4.82,2.8) (CCC) {$\bullet$};
        \node[red] at (3.82,3.8) (DDD) {$\bullet$};
        \draw (AAA) -- (BBB);
        \draw (AAA) -- (CCC);
        \draw (BBB) -- (DDD);
        \draw (CCC) -- (DDD);
        \node[olive] at (2.35,3.35) (M) {$\bullet$};
        \node[olive] at (2.35,2.25) (N) {$\bullet$};
        \draw (N) -- (M);
        \draw(AAA) -- (N);
        \draw (M) -- (DDD);
\draw[thick,black,->] (0.8,-0.1) -- (2.6,1.5); 
\draw[thick,black,->] (6.2,-0.1) -- (4.4,1.5); 
    \end{tikzpicture}
\caption{Amalgamating the $4$-element Boolean lattice with itself}\label{fig:4elattice}
\end{figure}

Then, by the amalgamation property, there is a lattice $L''\subseteq L$ that amalgamates $L$ with itself through $L'$ and the embeddings $f$ and $g$, hence $L''$ contains at least $6$ elements, see Figure~\ref{fig:4elattice}. This is a contradiction, a subalgebra of a $4$-element algebra cannot have more than $4$ elements. 

\section{Networks of monounary algebras}

In this section, we give a persuasive example of a network of large embeddings. This
is a network of easily visualizable structures of the form
$\a{A}=\langle A,f\rangle$ with one unary operation $f$. Such a
structure is called \emph{\textbf{monounary algebra}}. A nice feature of monounary algebras is that it is easy to describe
them pictorially. The algebra $\a{A}=\langle A,f\rangle$ can be represented by the directed graph $(A,E_A)$, with nodes $A$ and edges $E_A=\{(a,f(a)):a\in A\}$. For instance, Figure~\ref{ex-MU} gives pictorial graphical descriptions of some important examples of monounary algebras.

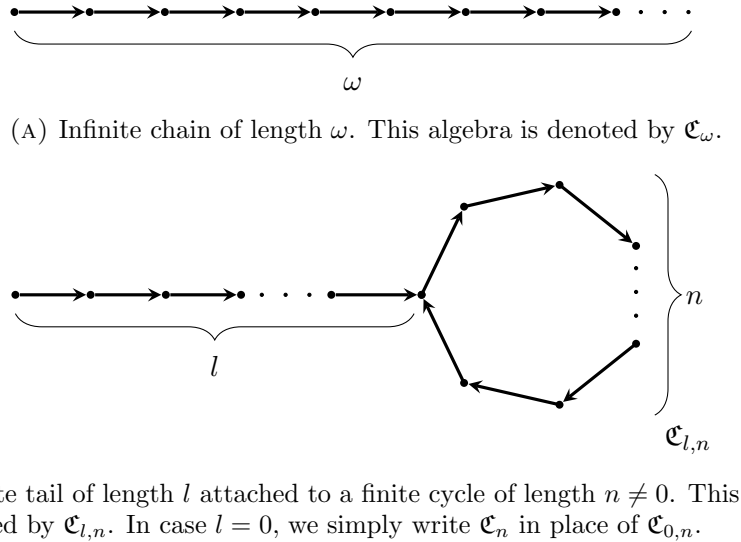
\begin{figure}[!ht]
\centering
\pgfmathsetmacro{\m}{8}
\begin{minipage}{0.75\linewidth}
\centering
\begin{tikzpicture}[decoration={markings,
  mark=between positions 0.2 and 0.8 step 9pt
  with { \draw [fill] (0,0) circle [radius=.6pt];}}]
\centering
\tikzset{arrow/.style = {very thick,->,> = stealth}}
\tikzset{point/.style = {shape=circle, fill, inner sep=0, minimum size=3}}

\begin{scope}[shift={(0,3)}]
\draw[white] (0,.5) to (7,.5);
\foreach \x in {0,1,...,\m}{
    \node[point] (\x) at (\x,0) {};
  \ifnum \x<\m \draw[arrow] (\x) to (\x+0.94,0);
	\fi
  }
   \path[postaction={decorate}] (\m) to (\m+1.4,0);

\draw [decorate,decoration={brace,amplitude=10pt}]
(9,-0.25) -- (0,-0.25) node [black,midway,yshift=-20pt] 
{$\omega$};
\end{scope}
\end{tikzpicture}
\subcaption{Infinite chain of length $\omega$. This algebra is denoted by $\MuN$.}\label{ex-MuN}
\end{minipage}

\bigskip

\begin{minipage}{0.75\linewidth}
\centering
\begin{tikzpicture}[decoration={markings,
  mark=between positions 0.2 and 0.8 step 9pt
  with { \draw [fill] (0,0) circle [radius=.6pt];}}]
  \centering
\tikzset{arrow/.style = {very thick,->,> = stealth}}
\tikzset{point/.style = {shape=circle, fill, inner sep=0, minimum size=3}}
\node[point] (0) at (0,0) {};
\node[point] (1) at (1,0){}; 
\node[point] (2) at (2,0){};
\node[point] (3) at (3,0){};
\node[point] (4) at (4.2,0){};
\node[point] (5) at (5.4,0){};
\draw[arrow] (0) to (1);
\draw[arrow] (1) to (2);
\draw[arrow] (2) to (3);
\path[postaction={decorate}] (3) to (4);
\draw[arrow] (4) to (5);
\foreach \a in {1,2,...,7}{
\node[point] (c\a) at ([shift={(6.9,0)}]180-\a*360/7: 1.5){};
\ifnum \a=4 \else \draw[arrow]  ([shift={(6.9,0)}]180-\a*360/7+360/7: 1.5) to  (c\a) ; \fi
}
\path[postaction={decorate}] (c3) to (c4);
\draw [decorate,decoration={brace,amplitude=10pt}]
(5.3,-0.25) -- (0,-0.25) node [black,midway,yshift=-20pt] 
{$l$};
\draw [decorate,decoration={brace,amplitude=10pt}]
(8.5,1.6) -- (8.5,-1.6) node [black,midway,xshift=15pt, yshift=-1pt] 
{$n$} node[below right]{$\Mpl{n}{l}$};
\end{tikzpicture}
\subcaption{Finite tail of length $l$ attached to a finite cycle of length $n\not=0$. This algebra is denoted by $\Mpl{n}{l}$. In case $l=0$, we simply write $\Core{n}$ in place of
$\Mpl{n}{0}$.}\label{ex-Mpl}
\end{minipage}

\caption{Pictorial examples of monounary algebras}\label{ex-MU}
\end{figure}

The theory of monounary algebras has an extensive literature, and the
research in this area is still active, see \eg \cite{JSP09, CJS18,
  Hetal11, JSP12, JPR12, KPSz19}. Before we begin our investigation,
we list some known facts about these algebras, see, \eg \cite{MMT87}. Throughout
this section, unless stated otherwise, $\CK$ is the class of all monounary
algebras and $\a{A}=\langle A,f\rangle$ is an arbitrary monounary algebra. 

%

We say that \emph{$\a{A}$ is \textbf{connected}} if for any $a,b\in A$ there are $k,m\in\omega$ such that $f^k(a)=f^m(b)$.\footnote{For every $n\in \omega$, we define $f^0(a)=a$ and $f^{n+1}(a)=f\left(f^n(a)\right)$.} Let $\a{B}=\langle B,g\rangle$ be a monounary algebra. Let
$\a{A}\sqcup\a{B}$ denote the \emph{\textbf{disjoint union} of the algebras} $\a{A}$ and $\a{B}$, whose universe $A\sqcup
B$ is the disjoint union of $A$ and $B$, and whose unary operation is the straightforward generalization of
the operations $f$ and $g$ to $A\sqcup B$. Finite and infinite disjoint unions of monounary algebras are defined as the natural generalizations of the operation $\sqcup$.


\begin{enumerate}
\item[(MU1)] Every monounary algebra can be uniquely decomposed into a disjoint union (finite or infinite) of connected monounary algebras.
\end{enumerate}

In other words, the monounary algebra $\a{A}$ is connected exactly if{}f its directed graph representation is connected in the sense of graph theory.  By (MU1), understanding the class of monounary algebras amounts to
understand its connected members. According to \cite[Thm.3.3]{MMT87}, we have the following characterization of connected monounary
algebras, where by \emph{an \textbf{in-tree}}, we understand a directed graph in which, for a particular vertex $r$ (called the \emph{\textbf{root}}) and any other vertex $u$, there is exactly one directed path from $u$ towards $r$. \footnote{An in-tree is essentially a rooted tree (as an undirected graph) in which all the edges are given an orientation (a direction) towards the root.}

\begin{enumerate}
\item[(MU2)] If $\a{A}$ is connected, then $\a{A}$ has a subalgebra that is isomorphic either to $\MuN$, or to $\Core{n}$ for some
  non-zero $n\in\omega$. This subalgebra is called a
\emph{\textbf{core} of $\a{A}$}, and it is unique up to isomorphism.\footnote{When no confusion is likely, we use phrase ``the core'' to mean the isomorphism type of the cores. It is worth noting that if a core is finite, then it is unique (not only up to isomorphism).} Deleting the edges of a core from the directed graph representation of $\a{A}$ would result in a union of disjoint in-trees each of which is rooted in an element of the core. See Figure~\ref{fig:mu3}.
\end{enumerate}

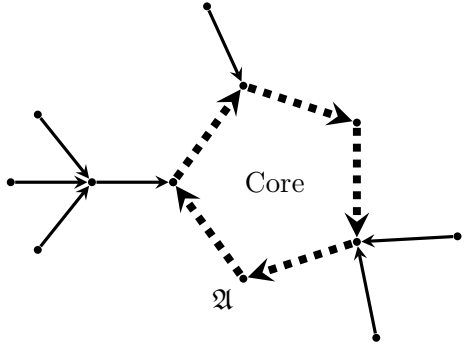
\begin{figure}[!ht]
  \centering
\begin{tikzpicture}[scale=0.9]
\tikzset{carrow/.style = {line width=3,->,> = stealth,shorten <=2,dashed}}
\tikzset{tarrow/.style = {very thick,->,> = stealth}}
\tikzset{point/.style = {shape=circle, fill, inner sep=0, minimum size=3}}
\node[point] (0) at (5.9,2.6) {};
\node[point] (1) at (8.4,-2.3){}; 
\node[point] (2) at (9.6,-0.8){};
\node[point] (3) at (3,0){};
\node[point] (4) at (4.2,0){};
\node[point] (5) at (5.4,0){};
\node[point] (4a) at (3.4,1) {};
\node[point] (4b) at (3.4,-1) {};
\foreach \a in {1,2,...,5}{
\node[point] (c\a) at ([shift={(6.9,0)}]180-\a*360/5: 1.5){};
\draw[carrow]  ([shift={(6.9,0)}]180-\a*360/5+360/5: 1.5) to  (c\a); 
}
\draw[tarrow] (0) to (c1);
\draw[tarrow] (1) to (c3);
\draw[tarrow] (2) to (c3);
\draw[tarrow] (3) to (4);
\draw[tarrow] (4) to (c5);
\draw[tarrow] (4a) to (4);
\draw[tarrow] (4b) to (4);
\node at (6.9,0) {Core};
\node[below left] at  (c4) {$\a{A}$};
\end{tikzpicture}
\caption{Core and in-trees rooted in the vertices of this core}\label{fig:mu3}
\end{figure}

Speaking of in-trees, we recall some related graph-theoretical notions. In an in-tree, \emph{the \textbf{parent} of a vertex $v$} is the vertex adjacent to $v$ in the path towards the root; every vertex has a unique parent except the root which has no parent. \emph{A \textbf{child} of a vertex $v$} is a vertex of which $v$ is the parent. A \emph{\textbf{leaf}} is a vertex has no children but has a parent (\ie it is not a root). Now, we will highlight the interconnection between the leafs of the in-trees quoted in (MU2) and the generators of $\a{A}$.

\emph{The monounary algebra $\a{A}$ is said to have an \textbf{infinite tail}} if there is an infinite sequence $$a_0,a_1,a_2,\ldots \in A$$
of elements of $A$ such that $a_n\not=a_m$ and $f(a_{n+1})=a_{n}$ for each $n,m\in \omega$ with $n\not=m$. \emph{An element $a\in A$ is called an \textbf{independent element}} if there is no $b\in A$ with $f(b)=a$. The set of all independent elements of $\a{A}$ is denoted by $\ind(\a{A})$. We denote the \emph{\textbf{minimum number of generators} of} $\a{A}$ by $\MGen(\a{A})$. If $\a{A}$ is not finitely
generated, then $\MGen(\a{A})$ is defined to be $\infty$.

\begin{enumerate}
\item[(MU3)] Assume that $\displaystyle\bigsqcup_{i\in I}\a{A}_i$ is a disjoint union of connected monounary algebras, for some set $I$, then 
    $$\MGen(\bigsqcup_{i\in I}\a{A}_i)=\sum_{i\in I}\MGen(\a{A}_i).$$
    So, in particular, if $\a{A}$ can be decomposed into an infinite disjoint union of connected monounary algebras, then $\MGen(\a{A})=\infty$. Now, suppose that $\a{A}$ is a connected algebra. One can effortlessly find a core for $\a{A}$ such that $\ind(\a{A})$ is the set of all leafs of all in-trees rooted in the vertices of this core, \cf (MU2). Whence, the following are legitimate.

\begin{enumerate}
    \item[(MU3a)] If $\a{A}$ has an infinite tail, then it cannot be finitely generated, \ie $\MGen(\a{A})=\infty$.
    \item[(MU3b)] If $\a{A}$ has no infinite tails and $\vert\ind(\a{A})\vert\geq 1$, then $\MGen(\a{A})=\vert \ind(\a{A})\vert$.\footnote{By $\vert\ind(\a{A})\vert$, we mean the size of $\ind(\a{A})$ if it is finite, and $\infty$ otherwise.}
    \item[(MU3c)] If $\a{A}$ has no infinite tails and $\ind(\a{A})=\emptyset$, then $\MGen(\a{A})=1$. In this case, $\a{A}$ is isomorphic to $\a{C}_n$ for some non-zero $n\in\omega$.
\end{enumerate}
\item[(MU4)] Every monounary algebra generated by a single element is
  isomorphic to either $\MuN$ or $\Mpl{n}{l}$ for some
  $l,n\in\omega$ and $n\neq 0$. See Figure~\ref{ex-MU}.
\end{enumerate}

\begin{prop}\label{embdgen}
The following is true for any monounary algebras $\a{A}$ and $\a{B}$:
$$\a{A}\text{ is embeddable into }\a{B}\implies \MGen(\a{A})\leq \MGen(\a{B}).$$
\end{prop}
\begin{proof}
It is enough to prove the proposition for connected monounary algebras $\a{A}$ and $\a{B}$. We can also assume that $\a{A}$ is a subalgebra of $\a{B}$. If $\MGen(\a{B})=\infty$, then there is nothing to prove. So, we can assume that $\MGen(\a{B})<\infty$. Thus, neither of $\a{A}$ and $\a{B}$ can have any infinite tail; otherwise (MU3a) would contradict the finiteness of $\MGen(\a{B})$. If $\ind(\a{A})=\emptyset$, then we are obviously done by (MU3c). Now, we assume that $\ind(\a{A})\not=\emptyset$. We will show that $\vert\ind(\a{A})\vert\leq \vert\ind(\a{B})\vert$. To this end, we define a function $\psi:\ind(\a{A})\rightarrow\ind(\a{B})$ as follows. Let $a\in\ind(\a{A})$ be an independent element of $\a{A}$. If $a$ is independent in $\a{B}$, then we define $\psi(a)=a$. If not, then there is $a_1\in \a{B}$ such that $f(a_1)=a$. If $a_1$ is an independent element in $\a{B}$, then we define $\psi(a)=a_1$. If not, then there is $a_2\in \a{B}$ such that $f(a_2)=a_1$. We can continue in this procedure till we find an independent element $a_n\in\a{B}$ such that $f^n(a_n)=a$. This procedure will halt because $\a{B}$ does not contain an infinite tail. Finally, we define $\psi(a)=a_n$. It remains to prove that $\psi$ is one-to-one. Let $a,b\in\ind(\a{A})$ be such that $\psi(a)=\psi(b)$. By the construction of $\psi$, there are $n,m\in\omega$ such that $a=f^n(\psi(a))$ and $b=f^m(\psi(b))$. Without loss of generality, we can assume that $n\leq m$. So, we can find a natural number $k\in\omega$ such that $n+k=m$. Since $f$ is a function, then the following is true in $\a{B}$:
$$b=f^m(\psi(b))=f^{n+k}(\psi(b))=f^k(f^n(\psi(b)))=f^k(f^n(\psi(a)))=f^k(a).$$ 
The equation $b=f^k(a)$ must be true in the subalgebra $\a{A}$ too because $\a{A}$ is closed under the operation $f$. That implies that $k=0$, because if $k>0$, then we will have a contradiction with the fact that both $a$ and $b$ are independent in $\a{A}$. Therefore, $a=b$ and $\psi$ is one-to-one as desired.
\end{proof}

Now we are ready to state our main result in this section. For each $n\in\omega$, $S_n$ denotes the set of all permutations on the set $\{i\in\omega:i< n\}$. From now on, the symbols $i,j,k$ will be used for indexes that vary among natural numbers in $\omega$. 

\begin{thm}\label{prop-perm}
  Let $n\le m$ be two positive natural numbers, let $\a{A}$ and $\a{B}$ be two
  monounary algebras having $n$-many and $m$-many connected components,
  and let $\a{A}=\displaystyle\bigsqcup_{i<n} \a{A}_i$ and
  $\displaystyle\a{B}=\bigsqcup_{j<m} \a{B}_j$ be their decomposition into
  disjoint unions of connected monounary algebras. Then
  \begin{equation}\label{eq-perm}
    \dis(\a{A},\a{B}) = \min_{\pi\in S_m}\left(\sum_{i<n}
    \dis\left(\a{A}_i,\a{B}_{\pi(i)}\right) +
    \sum_{n\leq k<m}\MGen\left(\a{B}_{\pi(k)}\right) \right).
  \end{equation}
\end{thm}

Roughly speaking, Theorem~\ref{prop-perm} states that in order to determine the distance between two monounary algebras $\a{A}$ and
$\a{B}$ having finitely many connected components it is enough to
determine the distance between their components, and then the distance
of $\a{A}$ and $\a{B}$ can be calculated by finding a ``best
matching'' of the components that minimizes the sums in equation~\eqref{eq-perm}. To prove Theorem~\ref{prop-perm}, we need first to prove some key propositions that seem interesting in their own.

The class $\CK$ of all monounary algebras is closed under the formation of subalgebras. So, in order to characterize the network of large embeddings in $\CK$, it is enough to consider the black edges that represent the symmetric closure of the large embedding relation. We start with characterizing the large embeddings between monounary algebras. To do so, we introduce a monounary algebra constructed from another one by attaching an
$n$-long ``tail'' at some element as follows.

\begin{defn}\label{def-MUtail}
  Let $\a{A}=\langle A,f\rangle$ be a monounary algebra, and let $a\in A$
  and let $n\in\omega$. We define a monounary algebra
  $\a{A}+_{a}n$ as follows: Let $\left\{x_i: i<n\right\}$ be a set of size
  $n$ of brand-new elements none of which is an element of $A$, and let
  \begin{equation*}
    \a{A}+_{a}n\defeq \left \langle A\sqcup\left\{x_i: i< n\right\},\hat
    f\right\rangle
  \end{equation*}
  where $\hat f$ is defined as
  \begin{equation*}
    \hat f(x)=\begin{cases}
     f(x)  &\text{if  }\ x\in A,\\
     x_{i+1}  &\text{if }\ x=x_i \text{ and } i< n-1, \text{ and}\\
     a  &\text{if }\ n\not=0 \text{ and } x=x_{n-1}.
    \end{cases}
  \end{equation*}
\end{defn}
Note that, for all $a,b\in A$ and $n,m\in\omega$, we have
\begin{equation}\label{eq-tail}
  \a{A}+_{a}0\cong\a{A} \quad \text{ and }\quad
  \a{A}+_{a}n+_{b}m\cong\a{A}+_{b}m+_{a}n.
\end{equation}

\begin{prop}\label{lem-largeMu}
  Let $\a{A},\a{B}$ be two monounary
  algebras. Then
\begin{equation}\label{eq-largeMu}
    \a{A}\add\a{B} \iff
    \begin{cases}
      \a{B}\cong \a{A}\sqcup\MuN,& \\
      \a{B}\cong \a{A}\sqcup\Mpl{n}{l} & \text{ for some }l,n\in\omega\text{ and } n\neq 0 \text{, or }  \\
      \a{B}\cong \a{A}+_{a}m & \text{ for some } m\in\omega \text{ and } a\in A.
    \end{cases}
\end{equation}
\end{prop}

\begin{proof}
  The direction ``$\Longleftarrow$'' is straightforward by (MU4) and the construction in Definition~\ref{def-MUtail}. To show the other direction ``$\Longrightarrow$'', we assume that $\a{A}\add\a{B}$, then there is a subalgebra $\a{A}'\subseteq\a{B}$ such that
  $\a{A}'$ is isomorphic to $\a{A}$, and $\a{B}=\langle A'\cup\{b\}\rangle$ for some
  $b\in B$. Let $g$ denote the unary operation of $\a{B}$.
  There are two cases:

\begin{enumerate}
\item[(I)] Suppose that $g^n(b)\not \in A'$ for every $n\in\omega$. In
  this case, the subalgebra of $\a{B}$ generated by $b$ must be disjoint from $A'$. By (MU4), $\langle
  b\rangle$ is either isomorphic to $\MuN$ or to $\Mpl{n}{l}$ for some appropriate
  natural numbers $n$ and $l$. Since $\langle A'\rangle\cong\a{A}$, this
  gives the first two cases of \eqref{eq-largeMu}.
\item[(II)] Suppose that $g^m(b)\in A'$ for some $m\in\omega$ such that
  $g^{m-1}(b)\not\in A'$ or $m=0$.  Then it is not hard to see that $\a{B}\cong\a{A}'+_{a'}m$,
  where $a'=g^m(b)$. Hence, there is $a\in A$ such that $\a{B}\cong\a{A}'+_{a'}m\cong\a{A}+_{a}m$, and this corresponds to the third case of \eqref{eq-largeMu}.\qedhere
\end{enumerate}
\end{proof}

The class of monounary algebras has the amalgamation property, \cf \eg
\cite[p.98]{KMPT83} and \cite{berth67}; hence, by Proposition~\ref{prop:ap->pup}, it also
has the push-up property. Below, we give a direct proof that gives insights on why $\CK$ has the push-up property.

\begin{prop}\label{prop-amalg}
  The class of monounary algebras has the push-up property.
\end{prop}

\begin{proof}
  We should show that, for all $\a{A},\a{B}\in\CK$,
  \begin{equation}\label{eq-amalg}
     (\exists \a{C}\in\CK)\quad \a{A}\dda \a{C}\add
     \a{B}\ \implies\ (\exists \a{D}\in\CK)\quad \a{A}\add \a{D}\dda
    \a{B}.
  \end{equation}
  By Proposition~\ref{lem-largeMu}, there are nine cases to be
  considered to prove \eqref{eq-amalg}. For the eight cases when either
  $\a{A}$ or $\a{B}$ can be achieved from $\a{C}$ by a disjoint union, the desired algebra $\a{D}$ can be easily constructed using disjoint union, \eg if
  $\a{A}\cong \a{C}\sqcup\MuN$, then $\a{D}=\a{B}\sqcup\MuN$, etc.
 
  Let us now assume that $\a{A}\cong \a{C}+_{c_1}n$ and $\a{B}\cong
  \a{C}+_{c_2}m$ for some natural numbers $n,m\in\omega$ and
  $c_1,c_2\in\a{C}$, and let $\alpha:\a{C}+_{c_1}n \to \a{A}$ and
  $\beta:\a{C}+_{c_2}m \to \a{B}$ be the corresponding isomorphisms. Let
  $a=\alpha(c_2)$ and $\a{D}=\a{A}+_{a}m$. Then, by
  \eqref{eq-tail}, it follows that $\a{D}\cong \a{B}+_{b}n$ where
  $b=\beta(c_1)$.
\end{proof}

Let us note that the converse of \eqref{eq-amalg} fails in the network
$\mathcal{N}(\CK)$ of monounary algebras because, for example, if
$\a{A}=\MuN$ and $\a{B}=\Mpl{n}{l}$ for some $n>0$ and $l$, then there
is no $\a{C}$ which is embeddable into both $\a{A}$ and $\a{B}$, but
both $\a{A}$ and $\a{B}$ are largely embeddable into
$\a{D}=\MuN\sqcup\Mpl{n}{l}$.  However, if we restrict the class $\CK$
to be the class of connected monounary algebras, then we have both
directions of \eqref{eq-amalg}.  In other words, the class of all
monounary algebras has the push-up property but lacks the push-down
property, while the class of connected monounary algebras has both
properties.

\begin{prop}\label{prop-diffcores}
  Let $\a{A}$ and $\a{B}$ be two connected monounary algebras. If
  $\a{A}$ and $\a{B}$ have isomorphic cores, then
  \begin{equation}\label{eq-samecore}
\dis(\a{A},\a{B})\le\MGen(\a{A})+\MGen(\a{B})-1.
  \end{equation}
If the cores of $\a{A}$ and $\a{B}$ are not isomorphic, then
  \begin{equation}\label{eq-diffcores}
\dis(\a{A},\a{B})=\MGen(\a{A})+\MGen(\a{B}).
  \end{equation}
\end{prop}

\begin{proof}
In general, it is true that 
\begin{equation}\label{eq-generalineq}
\dis(\a{A},\a{B})\le\MGen(\a{A})+\MGen(\a{B})  
\end{equation}
because $\a{B}\addn{\MGen(\a{A})}\a{A}\sqcup\a{B}$ and
$\a{A}\addn{\MGen(\a{B})}\a{A}\sqcup\a{B}$ if both $\a{A}$ and $\a{B}$
are finitely generated, and this inequality reduces to the trivial
$\dis(\a{A},\a{B})\le \infty$ otherwise.

Assume first that the cores of $\a{A}$ and $\a{B}$ are isomorphic. Now, we need to prove \eqref{eq-samecore}. If one of the algebras $\a{A}$ and $\a{B}$ is infinitely generated, then there is nothing to prove. Assume that both $\a{A}$ and $\a{B}$ are finitely generated. We
are going to prove \eqref{eq-samecore} by induction on
$k=\MGen(\a{A})+\MGen(\a{B})$.  

If $k=2$, then either
$\a{A}\cong\a{B}\cong\MuN$ or $\a{A}\cong\Mpl{n}{p}$ and
$\a{B}\cong\Mpl{n}{q}$ for some natural numbers $p$, $q$ and $n$. So,
we have $\dis(\a{A},\a{B})\le1$ in both cases. Inductively, assume
that \eqref{eq-samecore} is true for any two connected monounary
algebras $\a{A}'$ and $\a{B}'$ having isomorphic cores and satisfying
$\MGen(\a{A}')+\MGen(\a{B}')=N$ for some natural $N\geq 2$. Now, suppose that
$\MGen(\a{A})+\MGen(\a{B})=N+1$. Then at least one of the two
algebras, say $\a{A}$, has more than one generator. Thus, one can
find a subalgebra $\a{A}'$ of $\a{A}$ such that $\MGen(\a{A}')=\MGen(\a{A})-1$
and $\a{A}'\add \a{A}$. Hence, by the induction hypothesis, we
have $\dis(\a{A'},\a{B})\le \MGen(\a{A}')+\MGen(\a{B})-1$. Therefore,
$$\dis(\a{A},\a{B})\leq \dis(\a{A},\a{A}')+ \dis(\a{A}',\a{B})\le \MGen(\a{A}')+\MGen(\a{B})=
  \MGen(\a{A})+\MGen(\a{B})-1,$$
and this is what we wanted to show.

Assume now that the cores of $\a{A}$ and $\a{B}$ are not isomorphic.
To prove equation~\eqref{eq-diffcores}, by the observation in
\eqref{eq-generalineq}, it is enough to show that
$\dis(\a{A},\a{B})\ge\MGen(\a{A})+\MGen(\a{B})$. If
$\dis(\a{A},\a{B})=\infty$, then we have nothing to prove. So, let
us assume that $\dis(\a{A},\a{B})<\infty$. Then, by
Proposition~\ref{prop-amalg} and Proposition~\ref{prop:pushpath}, there
is a monounary algebra $\a{D}$ and $n,m\in\omega$ such that
$\dis(\a{A},\a{B})=n+m$, $\a{A}\addn{m}\a{D}$ and
$\a{B}\addn{n}\a{D}$. Hence, $\a{D}$ contains isomorphic images of both
$\a{A}$ and $\a{B}$. Since the cores of $\a{A}$ and $\a{B}$ are not
isomorphic, these images should be contained in different components
of $\a{D}$. Therefore, $\a{D}$ must also contain an isomorphic image
of $\a{A}\sqcup\a{B}$. Thus, $n\geq \MGen(\a{A})$ because by Proposition~\ref{embdgen} one have to add
at least an isomorphic copy of $\a{A}$ to $\a{B}$ in order to generate
$\a{D}$. Similarly, we must have $m\ge\MGen(\a{B})$. Consequently,
$\dis(\a{A},\a{B})=n+m\ge\MGen(\a{A})+\MGen(\a{B})$, and we are
done.
\end{proof}

Now, we will prove a formula that gives the generator distance between two monounary algebras (with finitely many connected components) in terms of the distance between their components. Let $\a{A}$ and $\a{B}$ be two monounary algebras, each of which has finitely many connected components. Assume that $\dis(\a{A},\a{B})$ is finite, then by Proposition~\ref{prop:pushpath} there are $p,q\in\omega$ and $\a{D}\in\CK$ such that
$$\a{A}\addn{p}\a{D}, \ \ \a{B}\addn{q}\a{D} \ \  \text{ and } \ \ \dis(\a{A},\a{B})=p+q.$$

Let $\a{A}-\a{D}-\a{B}$ denote the corresponding path from $\a{A}$ to $\a{B}$ through the algebra $\a{D}$.  Let $\alpha:\a{A}\to\a{D}$ and
$\beta:\a{B}\to\a{D}$ be the respective compositions of large embeddings
given by relations $\a{A}\addn{p}\a{D}$ and $\a{B}\addn{q}\a{D}$. We will refer to the condition $\dis(\a{A},\a{B})=p+q$ as \emph{the \textbf{minimality condition}}.

Note that two elements that belong to the same connected component of a monounary algebra must have their images under any homomorphism again in the same connected component. Hence, each component of $\a{A}$ and $\a{B}$ is
mapped into a connected component of $\a{D}$ by $\alpha$ and $\beta$. We
also have that $\alpha$ maps the different components of $\a{A}$ into
different components of $\a{D}$ because $\alpha$ is one-to-one, and the
core of a component of $\a{A}$ has to be mapped into the core of a component of $\a{D}$. The
same holds for $\beta$ and the components of $\a{B}$ for the same reasons. 

Thus, there are injective functions $\Psi_{\a{A}}$ and $\Psi_{\a{B}}$ from the components of $\a{A}$ and the components of $\a{B}$, respectively, into the components of $\a{D}$. For easy reference and visualization, we use a coloring style that codes these injective functions as follows.%

\begin{itemize}
\item[{\color{red}{$\bullet$}}] We color each component of $\a{D}$ that appeared in both the range of $\Psi_{\a{A}}$ and the range of $\Psi_{\a{B}}$ with a different shade of {\color{red}{red}}.
\item[{\color{blue}{$\bullet$}}] We color each component of $\a{D}$ that appeared in the range of $\Psi_{\a{A}}$ but not in the range of $\Psi_{\a{B}}$ with a different shade of {\color{blue}{blue}}.
\item[{\color{Green}{$\bullet$}}] We color each component of $\a{D}$ that appeared in the range of $\Psi_{\a{B}}$ but not in the range of $\Psi_{\a{A}}$ with a different shade of {\color{Green}{green}}.
\item[{\color{gray}{$\bullet$}}] We color each component of $\a{D}$ that appears neither in the range of $\Psi_{\a{B}}$ nor in the range of $\Psi_{\a{A}}$ with a different shade of {\color{gray}{gray}}.
\end{itemize}

By Proposition~\ref{lem-largeMu}, we know that each large embedding in $\a{A}\addn{p}\a{D}$ adds a piece of a component of the big algebra $\a{D}$ to the smaller algebra $\a{A}$; it adds either a new core or a tail to an existing component. The same also is true for each large embedding in $\a{B}\addn{q}\a{D}$. 

\begin{itemize}
    \item We color each large embedding step in $\a{A}-\a{D}-\a{B}$ with the same color of the component of $\a{D}$ to which this large embedding is contributing in its construction.
    \item We color each component of $\a{A}$ and $\a{B}$ with the same color of its image under the injective functions $\Psi_{\a{A}}$ and $\Psi_{\a{B}}$, respectively.
\end{itemize}

\begin{figure}[!htb]
    \centering
\begin{tikzpicture}[]
      \tikzstyle{bullet}=[draw,circle,,inner sep=3]
	\tikzstyle{steps}=[very thick, dash pattern=on 25pt off 10pt,postaction={decorate,
    decoration={markings,
    mark=between positions 29pt and 1 step 35pt with {\arrow{>};}}}, >=stealth]
	
	\tikzstyle{3steps}=[very thick, dash pattern=on 37pt off 9pt,postaction={decorate,
    decoration={markings,
    mark=between positions 41pt and 1 step 45pt with {\arrow{>};}}}, >=stealth]

	\tikzstyle{2steps}=[very thick, dash pattern=on 60pt off 10pt,postaction={decorate,
    decoration={markings,
    mark=between positions 64pt and 1 step 70pt with {\arrow{>};}}}, >=stealth]

\draw[dashed] (-4.93,0.1) to node[above left] {$\alpha$} (-0.93,3.1);
\draw[dashed] (1.93,-0.1) to (5.93,2.9);
  
\draw[dashed] (3.07,-0.1) to(-0.93,2.9);
\draw[dashed] (9.93,0.1) to  node[above right] {$\beta$}  (5.93,3.1);  	

      \draw (2.5,3)  ellipse [x radius=3.5,y radius=0.5]  node[above left=15]{$\a{D}$};
      \draw (-1.5,0)  ellipse [x radius=3.5,y radius=0.5]  node[left=100]{$\a{A}$};
      \draw (6.5,0)  ellipse [x radius=3.5,y radius=0.5]  node[right=100]{$\a{B}$};

      \foreach \x in {0,1,...,5}{
	      \coordinate (d\x) at (\x,3) ;
            \coordinate (b\x) at (\x+4,0) ;
               \coordinate (a\x) at (\x-4,0);
          }

	\node[bullet,blue,fill] (na1) at (a0) {};
	\node[bullet,blue!65,fill] (na2) at (a1) {};
	\node[bullet,red!85!black,fill] (na3) at (a2) {};
	\node[bullet,red,fill] (na4) at (a3) {};
    \node[bullet,green!42!black,very thick,dotted] (na5) at (a4) {};	
	\node[bullet,green!75!black,very thick,dotted] (na6) at (a5) {};	

	\node[bullet,blue,very thick, dotted] (nb00) at (b0) {};
	\node[bullet,blue!65,very thick,dotted] (nb0) at (b1) {};
	\node[bullet,red!85!black,fill] (nb1) at (b2) {};
	\node[bullet,red,fill] (nb2) at (b3) {};
	\node[bullet,green!42!black,fill] (nb3) at (b4) {};
	\node[bullet,green!75!black,fill] (nb4) at (b5) {};

	\node[bullet,blue,fill] (nd1) at (d0) {};
	\node[bullet,blue!65,fill] (nd2) at (d1) {};
	\node[bullet,red!85!black,fill] (nd3) at (d2) {};
	\node[bullet,red,fill] (nd4) at (d3) {};
	\node[bullet,green!42!black,fill] (nd5) at (d4) {};
	\node[bullet,green!75!black,fill] (nd6) at (d5) {};

	\draw[->, thick, red!85!black,shorten <=2, shorten >=2,3steps] (na3) to (nd3);     
	\draw[->, thick, red,shorten <=2, shorten >=2,2steps] (na4) to (nd4);   
	\draw[->, thick, green!42!black,shorten <=2, shorten >=2,2steps] (na5) to (nd5);     
    \draw[->, thick, green!75!black,shorten <=2, shorten >=2,steps] (na6) to (nd6);     

	\draw[->, thick, blue,shorten <=2, shorten >=2,steps] (nb00) to (nd1);     
 	\draw[->, thick, blue!65,shorten <=2, shorten >=2,steps] (nb0) to (nd2);    
	\draw[->, thick, red!85!black,shorten <=2, shorten >=2,steps] (nb1) to (nd3);     
	\draw[->, thick, red,shorten <=2, shorten >=2,3steps] (nb2) to (nd4);     
    
 \draw[red!85!black,thick, shorten <=5,shorten >=5, loosely dashed] (na3) to[out=-45, in=225] (nb1);
 \draw[red,thick, shorten <=5,shorten >=5, loosely dashdotdotted] (na4) to[out=-45, in=225] (nb2);   

\node at (2.5,-2) {$P$};
	
\end{tikzpicture}

    \caption{Coloring the components of the algebras $\a{A}$, $\a{B}$ and $\a{D}$, and the large embedding steps in the path $\a{A}-\a{D}-\a{B}$\label{fig-perm}}
\end{figure}
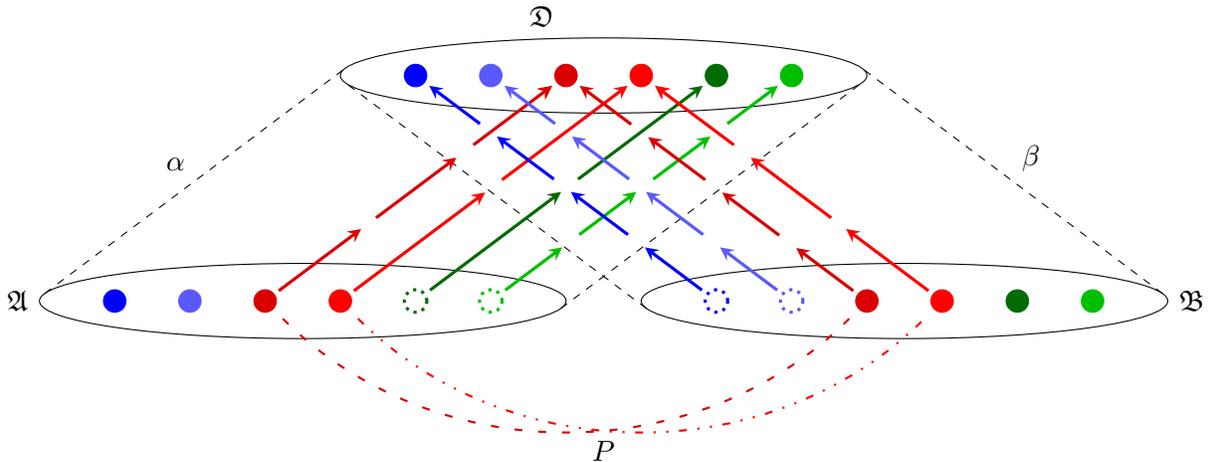

Note that there is no green component of $\a{A}$, and there is no blue component of $\a{B}$. It also turned out that there cannot be a gray component of $\a{D}$ and/or a gray large embedding in $\a{A}-\a{D}-\a{B}$, see Figure~\ref{fig-perm}. This is true because by deleting all the gray components and the gray large embeddings we would get a shorter path; which is a contradiction. Note that a large embedding step of some color affects only the component of $\a{D}$ that has the same color. This fact will be used quite frequently, so we give it a name. We call this feature \emph{the \textbf{independence of the colors}}.

Let $R$, $B$ and $G$ be the sets of all shades of red, blue and green, respectively, that were used in the above coloring. The component of $\a{A}$ which is colored by color $c$ is denoted by $\a{A}_c$. The same applies to the components of the monounary algebras $\a{B}$ and $\a{D}$.  

Let $b\in B$ be a blue color. The total number of the $b$-colored large embedding steps in $\a{B}\addn{q}\a{D}$ gives the minimum number of generators of $\a{D}_b$; this is the shortest possible way to build $\a{D}_b$ from scratch. If not, then one can use the independence of colors to derive a contradiction with the minimality condition. Similar conclusions can be drawn regarding the green components of $\a{D}$ and the green large embeddings in $\a{A}\addn{p}\a{D}$.

If $\a{A}_b$ is not isomorphic to $\a{D}_b$, then we can use Proposition~\ref{embdgen} and the independence of colors to shorten the path $\a{A}-\a{D}-\a{B}$ by replacing the component $\a{D}_b$ with a smaller one isomorphic to $\a{A}_b$ in the big algebra $\a{D}$, and this would contradict the minimality condition. Thus, there is no $b$-colored large embedding in $\a{A}\addn{p}\a{D}$. The same holds for the green components of $\a{B}$ and $\a{D}$. Hence, 

\begin{enumerate}
\item[(A)] For each $b\in B$ and each $g\in G$, we have
    \begin{align*}
    \text{the number of all $b$-colored large embedding steps in $\a{A}-\a{D}-\a{B}$} & =  \MGen(\a{D}_b),\\
    \text{the number of all $g$-colored large embedding steps in $\a{A}-\a{D}-\a{B}$} & =  \MGen(\a{D}_g).
    \end{align*}
\item[(B)] For each $b\in B$ and each $g\in G$, we have 
$$\a{A}_b\cong\a{D}_b \ \ \& \ \ \a{B}_g\cong \a{D}_g.$$
\end{enumerate}

Let $b\in B$ and $g\in G$. All the $b$-colored and the $g$-colored large embeddings in $\a{A}-\a{D}-\a{B}$ form a path between $\a{A}_b$ and $\a{B}_g$. This can be guaranteed by the independence of colors. By (A) and (B), the length of this path is $\MGen(\a{A}_b)+\MGen(\a{B}_g)$. On the other hand, this path should have the shortest length among all paths connecting $\a{A}_b$ and $\a{B}_g$ in $\mathcal{N}(\CK)$. Otherwise, by the independence of colors, one can construct a path between $\a{A}$ and $\a{B}$ that is shorter than $\a{A}-\a{D}-\a{B}$. Thus, Proposition~\ref{prop-diffcores} implies the observation (C) below.

\begin{enumerate}
    \item[(C)] For each $b\in B$ and each $g\in G$, $\a{A}_b$ and $\a{B}_g$ must have none isomorphic cores, and
    $$\dis(\a{A}_b,\a{B}_g)=\MGen(\a{A}_b)+\MGen(\a{B}_g).$$
\end{enumerate}   

It is time now to talk about the red components. Again, the next observation follows by the minimality condition and the independence of colors: 

\begin{enumerate}
    \item[(D)] For each $r\in R$, the cores of $\a{A}_r$, $\a{B}_r$ and $\a{D}_r$ must be isomorphic, and 
    $$\text{the number of all $r$-colored large embedding steps in $\a{A}-\a{D}-\a{B}$} = \dis(\a{A}_r,\a{B}_r).$$
\end{enumerate}  

Therefore, by (A), (B) and (C), we have 
\begin{equation}\label{eq:discomp}
    \dis(\a{A},\a{B})=\sum_{b\in B} \MGen(\a{A}_b) + \sum_{g\in G} \MGen(\a{B}_g) + \sum_{r\in R}\dis(\a{A}_r,\a{B}_r).
\end{equation}

Now, we are ready to prove our main Theorem.

\begin{proof}[Proof of Theorem~\ref{prop-perm}]
We recall the assumptions of the theorem. Let $n\le m$ be two natural numbers, let $\a{A}$ and $\a{B}$ be two
monounary algebras having $n$-many and $m$-many connected components, and let $\displaystyle\a{A}=\bigsqcup_{i<n} \a{A}_i$ and
$\displaystyle\a{B}=\bigsqcup_{j<m} \a{B}_j$ be their decomposition into disjoint unions of connected monounary algebras. We need to prove
that 
\begin{equation}\label{eq:distance}
\dis(\a{A},\a{B})=\min_{\pi\in
  S_m}\left(\sum_{i<n} \dis\left(\a{A}_i,\a{B}_{\pi(i)}\right)
+ \sum_{n\leq k<m}\MGen\left(\a{B}_{\pi(k)}\right) \right).
\end{equation}
It is easy to see that 
\begin{equation*}
\dis(\a{A},\a{B}) \le \sum_{i<n}
\dis\left(\a{A}_i,\a{B}_{\pi(i)}\right) +
\sum_{n\leq k<m}\MGen\left(\a{B}_{\pi(k)}\right) \ \ \ \text{ for all } \pi\in S_m.
\end{equation*}
Indeed, for $\pi\in S_m$, the algebra $\a{B}$ can be reached from $\a{A}$ by joining the
$\dis\left(\a{A}_i,\a{B}_{\pi(i)}\right)$-long paths going from $\a{A_i}$ to $\a{B_{\pi(i)}}$, for each $i< n$, and then continuing by building the remaining components of $\a{B}$ starting from scratch in $\displaystyle\sum_{n\leq k<m}\MGen\left(\a{B}_{\pi(k)}\right)$-many steps, see
Figure~\ref{fig-perm0}; let $\a{A}-\pi-\a{B}$ denote this path between
$\a{A}$ and $\a{B}$ which is based on  the permutation $\pi$.

 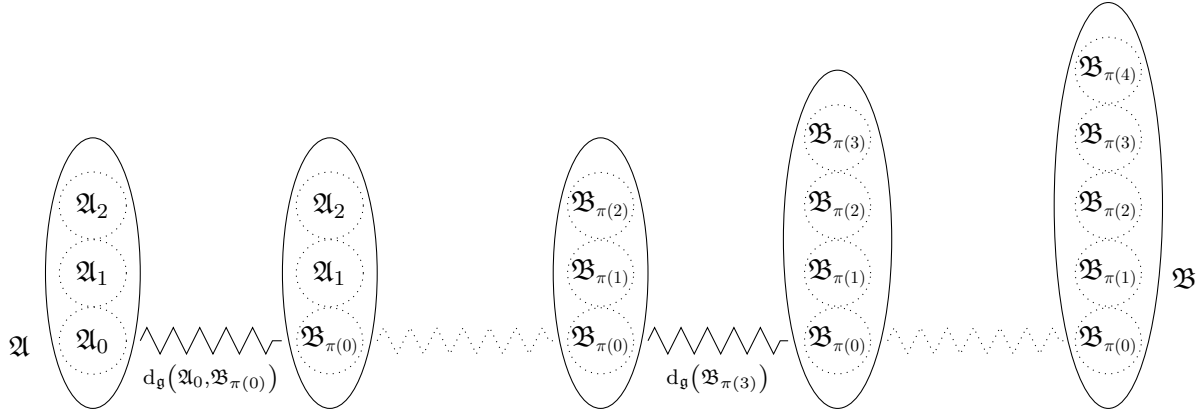
\begin{figure}
 \centering
 \begin{tikzpicture}[scale=0.9]
    \tikzstyle{bullet}=[draw,dotted,circle,inner sep=0]
    \tikzstyle{sago}=[decorate, decoration={zigzag, amplitude=5}, draw=black]

\draw (15,2)  ellipse [x radius=0.8,y radius=3]  node[below right=20]{$\a{B}$};
\draw (11,1.5)  ellipse [x radius=0.8,y radius=2.5]  node[above left=15]{};     
\draw (7.5,1)  ellipse [x radius=0.7,y radius=2]  node[below right=15]{};
\draw (3.5,1)  ellipse [x radius=0.7,y radius=2]  node[below right=15]{};
\draw (0,1)  ellipse [x radius=0.7,y radius=2]  node[below left=20]{$\a{A}$};

\foreach \x in {0,1,...,4}{
	\node[bullet] (b\x) at (15,\x) {$\a{B}_{\scalebox{0.6}{$\pi({\x})$}}$};
	\ifnum \x<4	     
		\node[bullet] (b0\x) at (11,\x)  {$\a{B}_{\scalebox{0.6}{$\pi({\x})$}}$} ;
	\fi
     \ifnum \x<3
		\node[bullet,inner sep =3.5] (a\x) at (0,\x)  {$\a{A}_\x$};
		 \ifnum \x>0 
			\node[bullet,inner sep =3.5] (a1\x) at (3.5,\x)  {$\a{A}_\x$};
		\else		
			\node[bullet] (a1\x) at (3.5,\x) {$\a{B}_{\scalebox{0.6}{$\pi({\x})$}}$};
		\fi
		\node[bullet] (a2\x) at (7.5,\x)  {$\a{B}_{\scalebox{0.6}{$\pi({\x})$}}$};	       
	 \fi
          }

\draw[sago] (.7,0) to node[below=5]{$\scriptstyle \dis\left(\a{A}_0,\a{B}_{\pi(0)}\right)$} (2.79,0);
\draw[sago,dotted] (4.2,0) to (6.85,0);
\draw[sago] (8.2,0) to node[below=5]{$\scriptstyle \dis\left(\a{B}_{\pi(3)}\right)$} (10.27,0);
\draw[sago,dotted] (11.74,0) to (14.4,0);
\end{tikzpicture}
 \caption{The path $\a{A}-\pi-\a{B}$ between $\a{A}$
   and $\a{B}$ based on the permutation $\pi$ \label{fig-perm0}}
\end{figure}

So we have that $\dis(\a{A},\a{B})$ is less than or equal to the minimum in the right-hand side of \eqref{eq:distance}. Consequently, if $\dis(\a{A},\a{B})=\infty$, we have nothing to prove. Let us assume that $\dis(\a{A},\a{B})<\infty$.  By Proposition~\ref{prop:pushpath}, there is a monounary algebra $\a{D}$ and $p,q\in\omega$ such
that 
$$\a{A}\addn{p}\a{D}, \ \ \a{B}\addn{q}\a{D} \ \  \text{ and } \ \ \dis(\a{A},\a{B})=p+q.$$
Now, we adopt the coloring system given in Figure~\ref{fig-perm} and the related discussion and terminology. We have two types of indexes for the components of $\a{A}$ and $\a{B}$; namely indexes induced by the component decomposition and indexes induced by the coloring. We define a correspondence $P$ between certain indexes $i<n$ of components of $\a{A}$ to that of $\a{B}$ as follows. The domain of $P$ is the set of all indexes of the red components of $\a{A}$. Let $\a{A}_i$ be a red component of $\a{A}$, then we define $P(i)=j$ where $j$ is the index of the component of $\a{B}$ that have exactly the same color as $\a{A}_i$. See Figure~\ref{fig-perm}. Pick a permutation $\pi\in S_m$ that extends the correspondence $P$. Such a permutation exists because $n\leq m$. For each $i< n\leq k< m$, we have
$$\a{A}_i \text{ is red }\Longleftrightarrow \a{B}_{\pi(i)} \text{ is red}, \ \ \ \ \a{A}_i \text{ is blue }\Longleftrightarrow \a{B}_{\pi(i)} \text{ is green}, \ \ \ \ 
\text{and } \ \ \ \ \a{B}_{\pi(k)} \text{ is green}.$$

Now, by \eqref{eq:discomp} and (C), we have 
\begin{align*} 
\sum_{i<n}
\dis\left(\a{A}_i,\a{B}_{\pi(i)}\right) & +
\sum_{n\leq k<m}\MGen\left(\a{B}_{\pi(k)}\right) \\
& = \sum_{\substack{i<n\\ \a{A}_i\text{ is blue}}}
\dis\left(\a{A}_i,\a{B}_{\pi(i)}\right) +
\sum_{\substack{i<n\\ \a{A}_i\text{ is red}}}
\dis\left(\a{A}_i,\a{B}_{\pi(i)}\right) +\sum_{n\leq k<m}\MGen\left(\a{B}_{\pi(k)}\right)\\
& = \sum_{\substack{i<n\\ \a{A}_i\text{ is blue}}}
d\left(\a{A}_i\right) + \sum_{\substack{i<n\\ \a{A}_i\text{ is blue}}}
d\left(\a{B}_{\pi(i)}\right) +\sum_{n\leq k<m}\MGen\left(\a{B}_{\pi(k)}\right)+\sum_{\substack{i<n\\ \a{A}_i\text{ is red}}}
\dis\left(\a{A}_i,\a{B}_{\pi(i)}\right)\\
&= \sum_{\substack{i<n\\ \a{A}_i\text{ is blue}}}
d\left(\a{A}_i\right) + \sum_{\substack{j<m\\ \a{B}_j\text{ is green}}}\MGen\left(\a{B}_{j}\right)+\sum_{\substack{i<n\\ \a{A}_i\text{ is red}}}
\dis\left(\a{A}_i,\a{B}_{\pi(i)}\right)\\
&= \sum_{b\in B} \MGen(\a{A}_b) + \sum_{g\in G} \MGen(\a{B}_g) + \sum_{r\in R}\dis(\a{A}_r,\a{B}_r)\\ & = \dis(\a{A},\a{B}).
\end{align*}

Therefore, $\dis(\a{A},\a{B})$ is bigger than or equal to the right-hand side of \eqref{eq:distance}, and we are done.
\end{proof}

Theorem~\ref{prop-perm} deals with the case when the two algebras have
finitely many components. In fact, it is not too difficult to see that for every two monounary algebras $\a{A}$ and $\a{B}$, either $\dis(\a{A},\a{B})=\infty$ or $\a{A}$ and $\a{B}$ differ only in finitely many connected components. This can be proved using the push-up property together with the fact that only one component can be affected by one step of large embedding, see Proposition~\ref{prop-amalg} and Proposition~\ref{lem-largeMu}. When $\a{A}$ and $\a{B}$ differ only in finitely many connected components, Theorem~\ref{prop-perm} can be used to reduce the distance between these algebras to the distance between their connected components.

Now, we can concentrate on the problem of determining the distance between connected monounary algebras. We already have the distance in the case when the algebras have non-isomorphic cores by Proposition~\ref{prop-diffcores}. Let us
consider the case of monounary algebras that have cores isomorphic to
$\Core{n}$ for some non-zero $n\in\omega$. We are going to have a reduction
step similar to what we have made in Theorem~\ref{prop-perm}. To do so, we introduce the notion of tree-algebra decomposition.

\begin{figure}[!ht]
  \centering
\begin{tikzpicture}[scale=0.9]

\tikzset{arrow/.style = {very thick,->,> = stealth}}
\tikzset{point/.style = {shape=circle, fill, inner sep=0, minimum size=3}}

\node[point] (0) at (5.9,2.6) {};
\node[point] (1) at (8.4,-2.3){}; 
\node[point] (2) at (9.6,-0.8){};
\node[point] (3) at (3,0){};
\node[point] (4) at (4.2,0){};
\node[point] (5) at (5.4,0){};
\node[point] (4a) at (3.4,1) {};
\node[point] (4b) at (3.4,-1) {};

\foreach \a in {1,2,...,5}{
\node[point] (c\a) at ([shift={(6.9,0)}]180-\a*360/5: 1.5){};
\draw[arrow]  ([shift={(6.9,0)}]180-\a*360/5+360/5: 1.5) to  (c\a); 
}

\draw[arrow] (0) to (c1);
\draw[arrow] (1) to (c3);
\draw[arrow] (2) to (c3);
\draw[arrow] (3) to (4);
\draw[arrow] (4) to (c5);
\draw[arrow] (4a) to (4);
\draw[arrow] (4b) to (4);

\node[below=3, xshift=2] at (c1) {0};
\node[below left] at (c2) {1};
\node[above left=2,yshift=-2] at (c3) {2};
\node[above=3,xshift=2] at (c4) {3};
\node[right=3] at (c5) {4};

\node[below left] at  (c4) {$\a{A}$};

\begin{scope}[shift={(12,0)}]
\node[point] (t0a) at (0,0.5) {};
\node[point] (t0b) at (0,-0.5) {};
\draw[arrow] (t0a) to (t0b);
\draw[arrow ] (t0b) to[out=-45, in=215, looseness=21] (t0b) node[below left=6]  {$\a{T}
_0$};

\node[point] (t1a) at (2,-0.5) {};
\draw[arrow ] (t1a) to[out=-45, in=215, looseness=21] (t1a) node[below left=6]  {$\a{T}_1$};

\node[point] (t2a) at (4,-0.5) {};
\node[point] (t2b) at (3.5,0.5) {};
\node[point] (t2c) at (4.5,0.5) {};

\draw[arrow] (t2b) to (t2a);
\draw[arrow] (t2c) to (t2a);
\draw[arrow ] (t2a) to[out=-45, in=215, looseness=21] (t2a) node[below left=6]  {$\a{T}_2$};

\node[point] (t3a) at (6,-0.5) {};
\draw[arrow ] (t3a) to[out=-45, in=215, looseness=21] (t3a) node[below left=6]  {$\a{T}_3$};

\node[point] (t4a) at (8,-0.5) {};
\node[point] (t4b) at (8,0.5) {};
\node[point] (t4c1) at  (7.3,1.3) {};
\node[point] (t4c2) at  (8,1.5) {};
\node[point] (t4c3) at  (8.7,1.3) {};

\draw[arrow] (t4b) to (t4a);
\draw[arrow] (t4c1) to (t4b);
\draw[arrow] (t4c2) to (t4b);
\draw[arrow] (t4c3) to (t4b);
\draw[arrow ] (t4a) to[out=-45, in=215, looseness=21] (t4a) node[below left=6]  {$\a{T}_4$};

\draw [decorate,decoration={brace,amplitude=10pt}]
(9,-1.5) -- (-1.5,-1.5) node [black,midway,yshift=-20pt] 
{tree-algebra decomposition of $\a{A}$};

\end{scope}

\end{tikzpicture}

  \caption{Tree-algebra decomposition of a monounary algebra with finite core\label{fig-treedecomp}}
\end{figure}
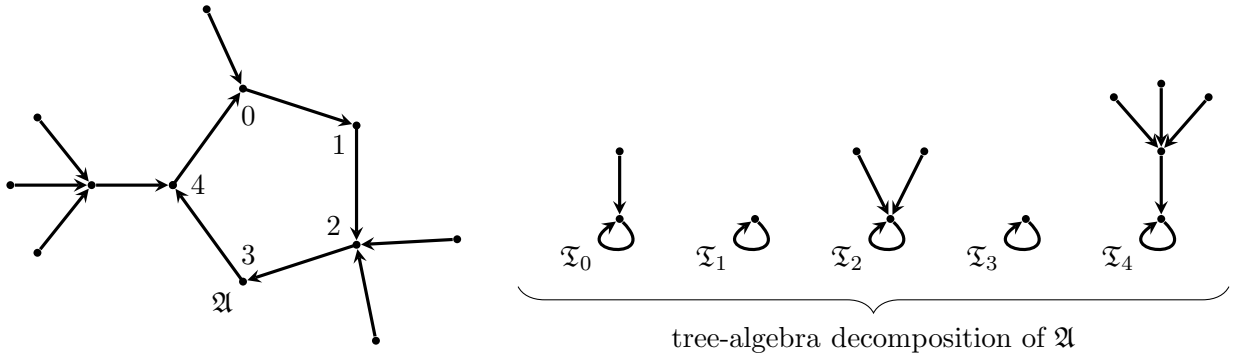

\begin{defn}
By a \emph{\textbf{tree-algebra}}, we mean a connected
monounary algebra whose core is isomorphic to $\Core{1}$. By a \emph{\textbf{forest-algebra}}, we mean a disjoint union of one or more tree-algebras.
\end{defn}

Let $\a{A}=\langle
A,f\rangle$ be a connected monounary algebra, and assume that its core is isomorphic to $\Core{n}$ for some
non-zero $n\in\omega$. Let $\langle T_i: i<n\rangle$ be an
enumeration of the in-trees whose roots $\langle r_i:i<n\rangle$ are elements of $\a{C}$ as described in (MU2). By a
\emph{\textbf{tree-algebra decomposition} of $\a{A}$}, we understand
the sequence $\langle \a{T}_i: i<n\rangle$ of tree-algebras,
where $\a{T}_i=\langle T_i,f_i\rangle$ is the tree-algebra
corresponding to in-tree $T_i$ such that
$f_i(r_i)=r_i$ and $f_i$ maps the other elements of $T_i$ to their
parents. See Figure~\ref{fig-treedecomp} for an
illustrative example.

\begin{prop}\label{prop-tree}
  Let $\a{A}=\langle A,f\rangle$ and $\a{B}=\langle B,g\rangle$ be two
  connected monounary algebras having cores isomorphic to $\Core{n}$,
  for some non-zero $n\in\omega$. Let $\langle \a{T}^A_i:i<n\rangle$ and $\langle \a{T}^B_i: i<n\rangle$ be the tree-algebra decompositions of $\a{A}$ and
  $\a{B}$, respectively.  Then
  \begin{equation*}
    \dis(\a{A},\a{B})=\min_{k<n} \sum_{\substack{i,j<n\\ j=i+k\mod
        n}} \dis\big(\a{T}^A_i,\a{T}^B_{j}\big).
  \end{equation*}
\end{prop}
\begin{proof}
 The proof goes exactly in the same spirit as the one of
 Theorem~\ref{prop-perm}; the main differences are that here we do not
 have unmatched pairs because $\a{A}$ and $\a{B}$ have the same number
 of components in their tree-algebra decompositions since they have
 isomorphic cores, and that here permutations matching tree-algebras
 in the two tree-algebra decomposition cannot be arbitrary because
 they also have to preserve the structure of the common finite core.
\end{proof}

Thus, the distance between connected monounary algebras having
isomorphic finite cores can be reduced to the distances between the
tree-algebras in their tree-algebra decompositions. Note that
Proposition~\ref{prop-tree} does not work for $\MuN$ among others
because monounary algebras of core $\MuN$ cannot be uniquely (up to permutation) decomposed into tree-algebras.

So far we have seen how to reduce the distance between monounary
algebras of finitely many components to the distances between their
connected components and how to reduce the distance between connected
algebras of finite cores to the distance between tree-algebras. Now we
are going to reduce the distance between finite tree-algebras to the
distance between shorter trees.

The idea roughly is to break a tree-algebra $\a{T}$ down into
the main-subtrees rooted in the children of the root of $\a{T}$. Formally, we do that by constructing a forest-algebra $\a{F}$ as
follows. We remove the core element $c$ together with the arrows
arriving to $c$ from the tree $\a{T}$ and then we attach a loop to
every element that was a child of $c$ in the tree-algebra, see
Figure~\ref{fig-forest} for an example illustration. The algebra
$\a{F}$ is called \emph{the \textbf{associated forest-algebra} (of
  main-subtrees) to $\a{T}$}.

\begin{figure}[!htb]
  \begin{tikzpicture}[scale=0.9]

\tikzset{arrow/.style = {very thick,->,> = stealth}}
\tikzset{point/.style = {shape=circle, fill, inner sep=0, minimum size=3}}

\node[point] (0) at (5,-0.5) {};
\draw[arrow,dashed] (0) to [out=-45, in=215, looseness=21]  (0);
\node[right,xshift=5] at (0) {$c$};

\node[point] (1) at (4,0.5){}; 
\node[point] (1a) at (3,1.5){}; 
\node[point] (1aa) at (2.5,2.5){};
\node[point] (1ab) at (3.5,2.5){};
\draw[arrow,dashed] (1) to (0); 
\draw[arrow] (1a) to (1);
\draw[arrow] (1aa) to (1a);
\draw[arrow] (1ab) to (1a);

\node[point] (2) at (5,0.5){};
\node[point] (2a) at (5,1.5){};
\node[point] (2aa) at (5,2.5){};
\draw[arrow,dashed] (2) to (0);
\draw[arrow] (2a) to (2);
\draw[arrow] (2aa) to (2a);

\node[point] (3) at (6,0.5){};
\node[point] (3a) at (6,1.5){};
\node[point] (3b) at (7, 1.5){};
\node[point] (3ba) at (6.5,2.5){};
\node[point] (3bb) at (7.5,2.5){};
\draw[arrow,dashed] (3) to (0);
\draw[arrow] (3a) to (3);
\draw[arrow] (3b) to (3);
\draw[arrow] (3ba) to (3b);
\draw[arrow] (3bb) to (3b);

\node[below left,xshift=-10,yshift=-10] at  (0) {$\a{T}$};

\begin{scope}[shift={(8,0)}]
\node[point] (t1) at (4,0.5){}; 
\node[point] (t1a) at (3,1.5){}; 
\node[point] (t1aa) at (2.5,2.5){};
\node[point] (t1ab) at (3.5,2.5){};
\draw[arrow] (t1a) to (t1);
\draw[arrow] (t1aa) to (t1a);
\draw[arrow] (t1ab) to (t1a);
\draw[arrow ] (t1) to[out=-45, in=215, looseness=21] (t1);
\end{scope}

\begin{scope}[shift={(10,0)}]
\node[point] (t2) at (5,0.5){};
\node[point] (t2a) at (5,1.5){};
\node[point] (t2aa) at (5,2.5){};
\draw[arrow] (t2a) to (t2);
\draw[arrow] (t2aa) to (t2a);
\draw[arrow ] (t2) to[out=-45, in=215, looseness=21] (t2);
\end{scope}

\begin{scope}[shift={(12,0)}]
\node[point] (t3) at (6,0.5){};
\node[point] (t3a) at (6,1.5){};
\node[point] (t3b) at (7, 1.5){};
\node[point] (t3ba) at (6.5,2.5){};
\node[point] (t3bb) at (7.5,2.5){};
\draw[arrow] (t3a) to (t3);
\draw[arrow] (t3b) to (t3);
\draw[arrow] (t3ba) to (t3b);
\draw[arrow] (t3bb) to (t3b);
\draw[arrow ] (t3) to[out=-45, in=215, looseness=21] (t3);
\end{scope}

\draw [decorate,decoration={brace,amplitude=10pt}] (19,0) -- (11,0)
node [black,midway,yshift=-20pt] {the associated forest-algebra
  $\a{F}$ (of main-subtrees) to $\a{T}$};

\end{tikzpicture}

   \caption{A finite tree-algebra broken down into shorter trees \label{fig-forest}}
\end{figure}
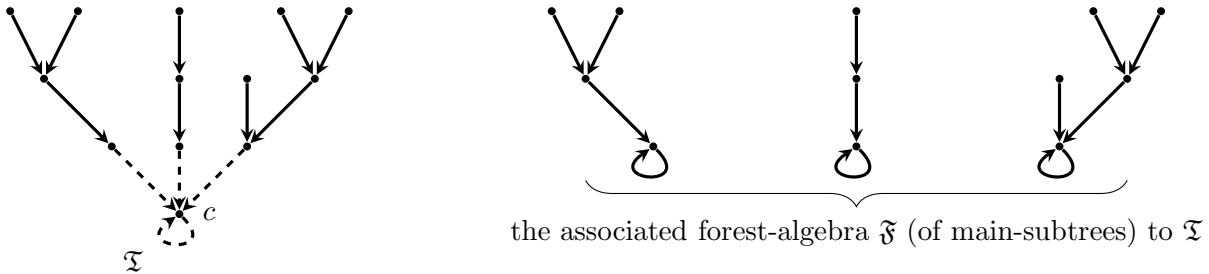

\begin{prop}\label{prop-treereduc}
Let $\a{T}_1$ and $\a{T}_2$ be two tree-algebras, and let $\a{F}_1$ and
$\a{F}_2$ be their associated forest-algebras. Then,
$\dis(\a{T}_1,\a{T}_2)=\dis(\a{F}_1,\a{F}_2)$.
\end{prop}

Consider two finite tree-algebras $\a{T}_1$ and $\a{T}_2$. By Proposition~\ref{prop-treereduc}, we know that the distance between $\a{T}_1$ and
$\a{T}_2$ is the same as the distance between their respected
associated algebras of main-subtrees $\a{F}_1$ and $\a{F}_2$. Now we
can use Theorem~\ref{prop-perm} to reduce the problem to finding
distances between pairs of shorter trees, namely the tree-algebras in
the tree-algebra decomposition of $\a{F}_1$ and $\a{F}_2$. So, all what remains is to prove Proposition~\ref{prop-treereduc}.

Before proving Proposition~\ref{prop-treereduc}, we make some observations concerning forest-algebras in general. Suppose that $\a{F}_1$ and $\a{F}_2$ are two forest-algebras of finite distance from each other. Then, by the push-up property, there are $p,q\in\omega$ and $\a{D}\in\CK$ such that
$$\a{F}_1\addn{p}\a{D}, \ \ \a{F}_2\addn{q}\a{D} \ \  \text{ and } \ \ \dis(\a{F}_1,\a{F}_2)=p+q.$$
One can apply the coloring (and the related discussion) of Figure~\ref{fig-perm} for $\a{F}_1$ and $\a{F}_2$ in the place of $\a{A}$ and $\a{B}$, respectively. Hence, by (B), (C) and (D), every component of the big algebra  $\a{D}$ must have a core isomorphic to $\a{C}_1$. Since every algebra in the path $\a{F}_1-\a{D}-\a{F}_2$ is embeddable into $\a{D}$, it follows that each component of each algebra in the path $\a{F}_1-\a{D}-\a{F}_2$ is a tree-algebra.

\begin{enumerate}
\item[(MF)] Two forest-algebras are of finite distance if{}f they are connected in the network $\mathcal{N}(\CK)$ by a minimal path that contains only forest-algebras.
\end{enumerate}

With a very similar argument, we can conclude the same for tree-algebras. 

\begin{enumerate}
\item[(MT)] Two tree-algebras are of finite distance if{}f they are connected in the network $\mathcal{N}(\CK)$ by a minimal path that contains only tree-algebras.
\end{enumerate}

Actually we can say more: Two connected monounary algebras $\a{A}$ and $\a{B}$ with isomorphic cores are of finite distance if{}f these algebras are connected in the network $\mathcal{N}(\CK)$ by a minimal path that contains only connected algebras with cores isomorphic to the cores of $\a{A}$ and $\a{B}$.

Let $F$ be the map taking a tree-algebra to its associated forest-algebra. It is not hard to see that, for every forest-algebra $\a{F}$, one can find a tree algebra $\a{T}$ such that $F(\a{T})=\a{F}$. This tree algebra is unique up to isomorphism. We pick one tree-algebra with this property and we denote it by $T(\a{F})$. The tree-algebra $T(\a{F})$ can be constructed by deleting all the loops from the tree-algebras in $\a{F}$ and joining all the roots of these trees to a brand new common root that has a loop. See Figure~\ref{fig-forest} ``backwards'' for an
illustration. Therefore,
\begin{equation}\label{12}
  F(T(\a{F}))\cong \a{F}\quad\text{and}\quad T(F(\a{T}))\cong \a{T}
\end{equation}
for all forest-algebra $\a{F}$ and tree-algebra $\a{T}$.  By their
definitions, it is clear that both $F$ and $T$ map isomorphic algebras
to isomorphic algebras. Moreover, by Proposition~\ref{lem-largeMu}, we have
\begin{equation}\label{13}
  \a{F}\add \a{F}' \iff T(\a{F})\add T(\a{F}') \qquad\text{and}\qquad
  \a{T}\add\a{T}'\iff F(\a{T})\add F(\a{T}').
\end{equation}

Now we are ready to prove Proposition~\ref{prop-treereduc}.
\begin{proof}[Proof of Proposition~\ref{prop-treereduc}]
   By (MT), a minimal path between $\a{T}_1$ and $\a{T}_2$ realizing
   $\dis(\a{T}_1,\a{T}_2)$ contains only tree-algebras. So, by \eqref{13}, taking
   the $F$-image of all of the tree-algebras in this minimal path would give a path of the same length between forest-algebras
   $\a{F}_1=F(\a{T}_1)$ and $\a{F}_2=F(\a{T}_2)$. Hence,
   $$\dis(\a{F}_1,\a{F}_2)\le\dis(\a{T}_1,\a{T}_2).$$
   Analogously, by (MF), a minimal path between $\a{F}_1$ and
   $\a{F}_2$ contains only forest-algebras. So, by \eqref{12} and \eqref{13}, taking the
   $T$-images of the forest-algebras in this path would give a path of
   the same length between the tree-algebras $\a{T}_1\cong T(\a{F}_1)$ and
   $\a{T}_2\cong T(\a{F}_2)$. Hence,
   $$\dis(\a{T}_1,\a{T}_2)\le\dis(\a{F}_1,\a{F}_2).$$ Consequently,
   $\dis(\a{T}_1,\a{T}_2)=\dis(\a{F}_1,\a{F}_2)$, and this is what we
   wanted to show.
\end{proof}

Now we give a concrete example. Consider algebras $\a{A}$ and $\a{B}$
illustrated in Figure~\ref{fig-muexample}.

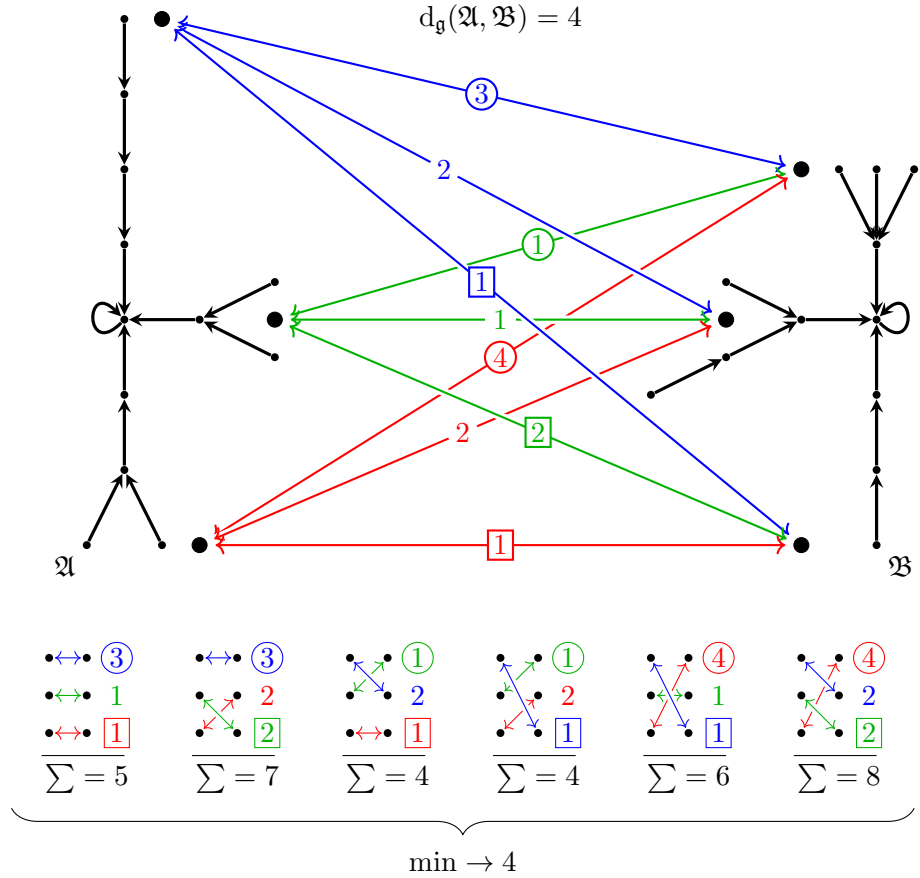
\begin{figure}[!hbt]
\begin{tikzpicture}[]
\tikzset{arrow/.style = {very thick,->,> = stealth}}
\tikzset{bullet/.style = {circle,draw,fill, inner sep=2}}
\tikzset{point/.style = {shape=circle,fill, inner sep=0, minimum size=3}}
\tikzset{cross line/.style={preaction={draw=white,-, shorten >=2pt, shorten
<=2pt, line width=3pt},<->,shorten >=0.5pt,shorten <= 0.5pt}}
\tikzset{big cross line/.style={preaction={draw=white,-, shorten >=2pt, shorten
<=2pt, line width=3pt},<->,shorten >=3pt,shorten <= 3pt}}
\tikzstyle{verda}=[green!70!black]

\tikzset{pics/perms/.style={code={
	\foreach \XX [count=\YY] in {#1} {
		\node[point] (T\YY) at (0,0.5*\YY) {};
		\node[point] (B\the\numexpr\XX) at (0.5,0.5*\YY) {};
   		}
	\foreach \XX [count=\YY] in {#1} {
	 \ifnum \YY=1 \draw[cross line,red] (T\YY) -- (B\YY); \fi
	 \ifnum \YY=2 \draw[cross line,verda] (T\YY) -- (B\YY); \fi
	 \ifnum \YY=3 \draw[cross line,blue] (T\YY) -- (B\YY); \fi
		}	
}}}

\node at (5,4) {$\dis(\a{A},\a{B})=4$};

\node[point] (0) at (-0,0) {};
\node[point] (1) at (1,0){}; 
\node[point] (11) at  (2,0.5){};
\node[point] (12) at (2,-0.5){};
\node[point] (2) at (0,1){};
\node[point] (21) at (0,2){};
\node[point] (211) at (0,3){};
\node[point] (2111) at (0,4){};
\node[point] (3) at (0,-1){};
\node[point] (31) at ((0,-2){};
\node[point] (311) at (-0.5,-3){};
\node[point] (312) at (0.5,-3){};

\draw[arrow] (0) to[out=-135, in=135, looseness=21] (0);
\draw[arrow] (1) to (0);
\draw[arrow] (11) to (1);
\draw[arrow] (12) to (1);
\draw[arrow] (2) to (0);
\draw[arrow] (21) to (2);
\draw[arrow] (211) to (21);
\draw[arrow] (2111) to (211);
\draw[arrow] (3) to (0);
\draw[arrow] (31) to (3);
\draw[arrow] (311) node[below left]{$\a{A}$} to (31);
\draw[arrow] (312)  to (31);

\node[point] (o) at (10,0){};
\node[point] (a) at (9,0){};
\node[point] (aa) at (8,0.5){};
\node[point] (ab) at (8,-0.5){};
\node[point] (aba) at (7,-1){};
\node[point] (b) at (10,1){};
\node[point] (ba) at (9.5,2){};
\node[point] (bb) at (10,2){};
\node[point] (bc) at (10.5,2){};
\node[point] (c) at (10,-1){};
\node[point] (ca) at (10,-2){};
\node[point] (caa) at (10,-3){};

\draw[arrow] (o) to[out=-45, in=45, looseness=21] (o);
\draw[arrow] (a) to (o);
\draw[arrow] (aa) to (a);
\draw[arrow] (ab) to (a);
\draw[arrow] (aba) to (ab);
\draw[arrow] (b) to (o);
\draw[arrow] (ba) to (b);
\draw[arrow] (bb) to (b);
\draw[arrow] (bc) to (b);
\draw[arrow] (c) to (o);
\draw[arrow] (ca) to (c);
\draw[arrow] (caa) node[below right] {$\a{B}$} to (ca);

\begin{scope}[shift={(-1,-6)}]
\pic{perms={1,2,3}};
\node[draw,inner sep=1,circle,blue,fill=white] at (0.9,1.5) {$3$};
\node[verda] at (0.9,1) {$1$};
\node[draw,inner sep=2,rectangle,red,fill=white]  at (0.9,0.5) {$1$};
\draw (-.1,0.2) to (1,0.2);
\node at (0.5,-0.1) {$\sum = 5$};
\end{scope}

\begin{scope}[shift={(1,-6)}]
\pic{perms={2,1,3}};
\node[draw,inner sep=1,circle,blue,fill=white] at (0.9,1.5) {$3$};
\node[red] at (0.9,1) {$2$};
\node[draw,inner sep=2,rectangle,verda,fill=white]  at (0.9,0.5) {$2$};
\draw (-.1,0.2) to (1,0.2);
\node at (0.5,-0.1) {$\sum = 7$};
\end{scope}

\begin{scope}[shift={(3,-6)}]
\pic{perms={1,3,2}};
\node[draw,inner sep=1,circle,verda, fill=white] at (0.9,1.5) {$1$};
\node[blue] at (0.9,1) {$2$};
\node[draw,inner sep=2,rectangle,red, fill=white]  at (0.9,0.5) {$1$};
\draw (-.1,0.2) to (1,0.2);
\node at (0.5,-0.1) {$\sum = 4$};
\end{scope}

\begin{scope}[shift={(5,-6)}]
\pic{perms={3,1,2}};
\node[draw,inner sep=1,circle,verda, fill=white] at (0.9,1.5) {$1$};
\node[red] at (0.9,1) {$2$};
\node[draw,inner sep=2,rectangle,blue ,fill=white]  at (0.9,0.5) {$1$};
\draw (-.1,0.2) to (1,0.2);
\node at (0.5,-0.1) {$\sum = 4$};
\end{scope}

\begin{scope}[shift={(7,-6)}]
\pic{perms={3,2,1}};
\node[draw,inner sep=1,circle,red,fill=white] at (0.9,1.5) {$4$};
\node[verda] at (0.9,1) {$1$};
\node[draw,inner sep=2,rectangle,blue, fill=white]  at (0.9,0.5) {$1$};
\draw (-.1,0.2) to (1,0.2);
\node at (0.5,-0.1) {$\sum = 6$};
\end{scope}

\begin{scope}[shift={(9,-6)}]
\pic{perms={2,3,1}};
\node[draw,inner sep=1,circle,red,fill=white] at (0.9,1.5) {$4$};
\node[blue] at (0.9,1) {$2$};
\node[draw,inner sep=2,rectangle,verda,fill=white]  at (0.9,0.5) {$2$};
\draw (-.1,0.2) to (1,0.2);
\node at (0.5,-0.1) {$\sum = 8$};
\end{scope}

\draw [decorate,decoration={brace,amplitude=10pt}]
(10.5,-6.5) -- (-1.5,-6.5) node [black,midway,yshift=-20pt] 
{$\min\to 4$};

\node[bullet] (X1) at (0.5, 4) {};
\node[bullet] (X2) at (2,0) {};
\node[bullet] (X3) at (1,-3) {};

\node[bullet] (Y1) at (9,2) {};
\node[bullet] (Y2) at (8,0) {};
\node[bullet] (Y3) at (9,-3) {};

\draw[thick,cross line,big cross line,blue] (X1) -- node[draw,inner sep=1,circle,fill=white] {3} (Y1); 
\draw[thick,cross line,big cross line,verda] (X2) -- node[draw,inner sep=1,circle,fill=white] {1} (Y1); 
\draw[thick,cross line,big cross line,red] (X3) -- node[draw,inner sep=1,circle,fill=white] {4} (Y1); 

\draw[thick,cross line,big cross line,blue] (X1) -- node[inner sep=1,circle,fill=white] {2} (Y2); 
\draw[thick,cross line,big cross line,verda] (X2) -- node[inner sep=1,circle,fill=white] {1} (Y2); 
\draw[thick,cross line,big cross line,red] (X3) -- node[inner sep=1,circle,fill=white] {2} (Y2); 

\draw[thick,cross line,big cross line,blue] (X1) -- node[draw,inner sep=2,rectangle,fill=white] {1} (Y3); 
\draw[thick,cross line,big cross line,verda] (X2) -- node[draw,inner sep=2,rectangle,fill=white] {2} (Y3); 
\draw[thick,cross line,big cross line,red] (X3) -- node[draw,inner sep=2,rectangle,fill=white] {1} (Y3); 

\end{tikzpicture}
  \caption{The distance between two finite tree-algebras \label{fig-muexample}}
\end{figure}

There are three main-subtrees of both algebras $\a{A}$ and $\a{B}$ of
the figure. To calculate the distance between $\a{A}$ and $\a{B}$, it
is enough to calculate the sum of the distances between these
main-subtrees for all the six permutations mapping these subtrees to
each other and take the minimum. This follows from
Proposition~\ref{prop-treereduc} and Theorem~\ref{prop-perm}. To
calculate the distances between any pair of the main-subtrees, we can
again use Proposition~\ref{prop-treereduc} and Theorem~\ref{prop-perm}
in the same way. However, since these main-subtrees are so simple, one
can easily determine their distances. Therefore, in the figure, we
just give their distances at the corresponding permutation without
further explanation. Since $4$ is the smallest sum, we have
$\dis(\a{A},\a{B})=4$.

The above considerations give a method by which one can calculate the
distance between any two finite monounary algebras. It is clear
however that this method is not very effective because one have to
consider lots of permutations in every reduction step. It is a task
for further research to find more effective algorithms and to
investigate the computational complexity of the problem of determining
the generator distance between two finite monounary algebras.

\begin{prob}
  What is the computational complexity of calculating the distance
  between two finite monounary algebras?
\end{prob}



Even though, Proposition~\ref{prop-treereduc} does not require the
tree-algebras to be finite, it does not seem to be of a great
help in the case when one of them is infinite. It is not
hard though to prove that two arbitrary tree-algebras are of finite
distance if{}f they differ only in finitely many branches, \ie if
both can be embedded into a tree-algebra that differs from each of
them only by finitely many extra elements. But even in this case, we
do not have a method to calculate this distance yet. We also remind
the reader that determining the distance between two connected
monounary algebras, each of which has a core isomorphic to $\MuN$, is
still open.


\begin{prob}
Determine the distance between countable monounary algebras.
\end{prob}

%

\bibliographystyle{plain}

\end{document}